\documentclass[12pt]{amsart}

\usepackage{graphicx}
\usepackage{amssymb}
\usepackage{amsthm}
\usepackage{bbold}
\usepackage{listings}
\usepackage{lineno}
\usepackage[margin=3cm]{geometry}
\usepackage[all,cmtip, color,matrix,arrow]{xy}
\usepackage{amsaddr}
\usepackage{tikz-cd}
\usepackage{amsmath}
\usepackage[utf8]{inputenc}
\usepackage{hyperref}
\usepackage[capitalize]{cleveref}
\crefname{thm}{Theorem}{Theorems}
\usepackage[export]{adjustbox}
\usepackage{todonotes}
\usepackage{bm}
\usepackage{wrapfig}
\usepackage{float}
\usepackage{mathtools}
\usepackage{aliascnt}
\newaliascnt{eqfloat}{equation}
\newfloat{eqfloat}{h}{eqflts}
\floatname{eqfloat}{Equation}
\usepackage{dirtytalk}
\usepackage[mathscr]{euscript}

\newcommand*{\ORGeqfloat}{}
\let\ORGeqfloat\eqfloat
\def\eqfloat{%
  \let\ORIGINALcaption\caption
  \def\caption{%
    \addtocounter{equation}{-1}%
    \ORIGINALcaption
  }%
  \ORGeqfloat
}

\makeatletter
\providecommand*{\shuffle}{%
  \mathbin{\mathpalette\shuffle@{}}%
}
\newcommand*{\shuffle@}[2]{%
  \sbox0{$#1\vcenter{}$}%
  \kern .15\ht0 
  \rlap{\vrule height .25\ht0 depth 0pt width 2.5\ht0}%
  \raise.1\ht0\hbox to 2.5\ht0{%
    \vrule height 1.75\ht0 depth -.1\ht0 width .17\ht0 %
    \hfill
    \vrule height 1.75\ht0 depth -.1\ht0 width .17\ht0 %
    \hfill
    \vrule height 1.75\ht0 depth -.1\ht0 width .17\ht0 %
  }%
  \kern .15\ht0 
}
\makeatother

\theoremstyle{definition}
\newtheorem{thm}{Theorem}[section]
\newtheorem{prop}[thm]{Proposition}
\newtheorem{lm}[thm]{Lemma}

\newtheorem{obs}[thm]{Observation}
\newtheorem{defin}[thm]{Definition}
\newtheorem{smpl}[thm]{Example}

\newtheorem{exe}[thm]{Exercise}
\newtheorem{const}[thm]{Construction}
\newtheorem{prob}[thm]{Problem}
\newtheorem{conj}[thm]{Conjecture}

\crefname{lm}{Lemma}{Lemmas}
\crefname{thm}{Theorem}{Theorems}
\crefname{prop}{Proposition}{Propositions}
\crefname{defin}{Definition}{Definitions}
\crefname{rem}{Remark}{Remarks}

\newcommand{\R}{\mathbb{R}}
\newcommand{\Z}{\mathbb{Z}}
\newcommand{\F}{\mathbb{F}}

\newcommand{\K}{\mathbb{K}}
\newcommand{\PP}{\mathbb{P}}

\newcommand{\one}{\mathbb{1}}

\newcommand{\CC}{\mathcal C}
\newcommand{\JJ}{\mathcal J}
\newcommand{\II}{\mathcal I}
\newcommand{\FF}{\mathcal F}

\newcommand{\vv}{\mathsf{v}}
\newcommand{\vw}{\mathsf{w}}
\newcommand{\vj}{\mathsf{j}}
\newcommand{\vx}{\mathsf{x}}
\newcommand{\vy}{\mathsf{y}}
\newcommand{\vz}{\mathsf{z}}
\newcommand{\va}{\mathsf{a}}
\newcommand{\ve}{\mathsf{e}}

\newcommand{\spn}{\mathrm{span}}
\newcommand{\rk}{\mathrm{rk}}
\newcommand{\tr}{\mathrm{tr}}
\newcommand{\Id}{\mathrm{Id}}
\newcommand{\conv}{\mathrm{conv}}
\newcommand{\maxcut}{\mathrm{maxcut}}
\newcommand{\we}{\mathrm{we}}
\newcommand{\spec}{\mathrm{spec} }

\begin{document}

\title{Lecture notes on algebraic methods in combinatorics} 

\author{Raul Penaguiao}
\email{raul.penaguiao@mis.mpg.de}
\address{Max Planck Institute for the Sciences Leipzig}
\subjclass[2010]{}
\date{Spring semester 2023} 

\begin{abstract}
These are lecture notes of a course taken in Leipzig 2023, spring semester.
It deals with extremal combinatorics, algebraic methods and combinatorial geometry.
These are not meant to be exhaustive, and do not contain many proofs that were presented in the course.
\end{abstract}

\maketitle

\tableofcontents

\section{Preliminary definitions}

We denote the set $\{1, \dots, n\}$ by $[n]$.

\subsection{Binomial coefficient bounds}

\begin{lm}\label{lm:binom_bound}
$$\binom{n}{k} \leq e^k \left( \frac{n}{k} \right)^k \, . $$
\end{lm}

\subsection{Fields}
Whenever $q$ is a power of a prime, we denote the unique field of cardinality $q$ by $\F_q$.
When $q$ is a prime number, we identify $\F_q$ with $\Z/_{q\Z}$.

\subsection{Vector spaces}

An \textbf{eigenvalue} $\lambda$ of a matrix $A$ with entries in $\K$ is a scalar such that $A - \lambda \Id $ is non-singular.
A non-zero vector $\vv $ in $\ker (A - \lambda \Id)  $ is said to be a corresponding \textbf{eigenvector}.

\section{Combinatorics}

\subsection{Clubs with rules}

\begin{prop}[Oddtowns]\label{prop:oddtown}
Let $E$ be a finite set.
A family $\CC \subset 2^E $ is called an \textbf{oddtown} of $E$ if
\begin{itemize}
\item Every $C \in \CC $ has an odd number of elements.

\item For two distinct $C, D \in \CC$, the set $C\cap D$ has an even number of elements.
\end{itemize}

Then, for any oddtown we have $|\CC| \leq |E|$
\end{prop}

\begin{proof}
We work in $\F_2^E$.
For each $C \in \CC$, let $\vv_C$ be the vector in $\F_2^E$ such that 
$$ (\vv_C)_i =\begin{cases*}
      & 1 \text{ if $i \in C$,}\\
      & 0 \text{ otherwise.}
    \end{cases*} $$

The first and second oddtown conditions give us respectively $\vv_C \cdot \vv_C = 1$ and $\vv_C \cdot \vv_D = 0 $ for any distinct $C, D \in \CC$.
Denote by $\mathbb{0}$ the all zero vector.

We claim that $\{\vv_C\}_{C \in \CC}$ is a linearly independent set in $\F_2^E$.
It follows imediately that $|\CC| \leq |E|$.
Indeed, if $\sum_{C \in \CC} \alpha_C \vv_C = \mathbb{0}$ for scalars $\alpha_C \in \F_2$, then
$$ 0 = \mathbb{0} \cdot \vv_D = \sum_{C \in \CC} \alpha_C \vv_C\cdot \vv_D = \alpha_D\, ,$$
for any $D \in \CC$, concluding the proof.
\end{proof}

\begin{prop}[Separated collections]
Let $E$ be a finite set.
A family $\CC \subset 2^E $ is called \textbf{separated} if for any two disjoint non-empty subfamilies $\II, \JJ \subseteq \CC$ we have $\bigcup_{I\in\II} I \neq \bigcup_{J\in\JJ} J$.

Then, for any separated family we have $|\CC| \leq |E|$.
\end{prop}

\begin{proof}
This time we work in $\R^E$.
For each $C \in \CC$, let $\vv_C$ be the vector in $\R^E$ such that 
$$ (\vv_C)_i =\begin{cases*}
      & 1 \text{ if $i \in C$,}\\
      & 0 \text{ otherwise.}
    \end{cases*} $$

We will show that $ \{ \vv_C\}_{C\in \CC}$ is a linearly independent set in $\R^E$.
That $|\CC| \leq |E|$ follows imediately.

Indeed, assume that $\sum_{C \in \CC} \alpha_C \vv_C = 0$.
Let $\II = \{C \in \CC | \alpha_C > 0\}$ and $\JJ = \{C \in \CC | \alpha_C < 0\}$.
Rearranging the equation above we have
$$\vv \coloneqq \sum_{C \in \II} \alpha_C \vv_C  = \sum_{C \in \JJ} (- \alpha_C) \vv_C\, .$$
Let $K = \{i \in E| \vv_i \neq 0\}$.
We have both that $\bigcup_{I\in\II} I = K$ and $\bigcup_{J\in\JJ} J = K$.
Also, $\II$ and $\JJ$ are disjoint.

This contradicts the fact that $\CC$ is separated unless $\II$ and $\JJ$ are empty.
But $\II$ and $\JJ$ empty implies $\alpha_C = 0 $ for any $C \in \CC$.
This shows that $ \{ \vv_C\}_{C\in \CC}$ is a linearly independent set, concluding the proof.
\end{proof}

\begin{prop}[Lindstorm's lemma]
Let $E$ be a finite set.
A family $\CC \subset 2^E $ is called \textbf{weakly separated} if for any two disjoint subfamilies $\II, \JJ \subseteq \CC$ we have $\bigcup_{I\in\II} I \neq \bigcup_{J\in\JJ} J$ or $\bigcap_{I\in\II} I \neq \bigcap_{J\in\JJ} J$.

Then, for any weakly separated family we have $|\CC| \leq |E| + 1$.
\end{prop}

\begin{proof}
This time we work in $\R^{E \uplus \{ \star \}}$.
For each $C \in \CC$, let $\vv_C$ be the vector in $\R^{E \uplus \{ \star \}}$ such that 
$$ (\vv_C)_i =\begin{cases*}
      & 1 \text{ if $i \in C$ or if $i = \star$,}\\
      & 0 \text{ otherwise.}
    \end{cases*} $$

We will show that $\{ \vv_C\}_{C\in \CC}$ is a linearly independent set in $\R^{E \uplus \{ \star \}}$.
That $|\CC| \leq |E|+1$ follows imediately.

Indeed, as before assume that $\sum_{C \in \CC} \alpha_C \vv_C = 0$.
Let $\II = \{C \in \CC | \alpha_C > 0\}$ and $\JJ = \{C \in \CC | \alpha_C < 0\}$.
Rearranging the equation above we have
$$\vv \coloneqq \sum_{C \in \II} \alpha_C \vv_C  = \sum_{C \in \JJ} (- \alpha_C) \vv_C\, .$$
Let $K = \{i \in E| \vv_i \neq 0\}$.
We have both that $\bigcup_{I\in\II} I = K$ and $\bigcup_{J\in\JJ} J = K$.
Also, $\II$ and $\JJ$ are disjoint.
Furthermore, $\vv_{\star } = \sum_{I \in \II} \alpha_I = - \sum_{J \in \JJ} \alpha_J$.
We can see that $i \in \bigcap_{I \in \II} I$ if and only if $\vv_i = \vv_{\star }$.
Similarly, $j \in \bigcap_{J \in \JJ} J$ if and only if $\vv_j = \vv_{\star }$.
We conclude that $\bigcap_{I \in \II} I = \bigcap_{J \in \JJ} J$.

This contradicts the fact that $\CC$ is weakly separated unless $\II$ and $\JJ$ are non-empty.
This shows that $ \{ \vv_C\}_{C\in \CC}$ is a linearly independent set, concluding the proof.
\end{proof}

\begin{prop}[Fischer's inequality]\label{prop:fischer}
Let $E$ be a finite set.
A family $\CC \subseteq 2^E$ is said to be $\lambda$\textbf{-Fischer} if $|C\cap D| = \lambda$ for all distinct $C, D \in \CC$.

Then, for any $\lambda$-Fischer family, if $\lambda \neq 0$ then $|\CC| \leq |E|$.
\end{prop}

\begin{proof}
We first deal with the case where $|C| = \lambda $ for some $C \in \CC$.
Then $\II \coloneqq \{ D \setminus C | D \in \CC \setminus \{C\} \}$ is a family of disjoint sets in $E \setminus C$, therefore
$$ |\II| \leq |E\setminus C| = |E| - \lambda \leq |E| - 1\, .$$
We conclude that $|\CC| \leq |E|$.

Now assume that $|C| > \lambda $ for all $C \in \CC$.
For $C \in \CC$, define the vectors $\vv_C$ in $\R^E$ as 
$$ (\vv_C)_i =\begin{cases*}
      & 1 \text{ if $i \in C$,}\\
      & 0 \text{ otherwise.}
    \end{cases*} $$
Let $A$ be the $|E| \times |\CC|$ matrix with column vectors $\vv_C$.

We write $\one $ for the all one vector with $|\CC|$ entries, which we also interpret as a $|\CC| \times 1$ matrix.
In this way, $\one \one^T$ is the all one matrix.
We have that $A^T A  = \lambda \one \one^T + \mathrm{diag}((d_C)_{C\in\CC})$, where $\mathrm{diag}((d_C)_{C\in\CC})$ is a diagonal matrix with entries $d_C$.
Because $|C| > \lambda$ for all $C \in \CC$, each $d_C$ is a positive integer.

We now show that $A^T A$ is full rank.
Because this is a square matrix, we is equivalently show that it is non-singular.
Assume that $A^T A \vx = 0$, let $s = \one^T \vx = \sum_{C \in \CC} \vx_C$.
Then, from the equation above and $d_i > 0$, we have
\begin{align*}
A^T A \vx &= \lambda s \one + \mathrm{diag}((d_C)_{C\in\CC}) \vx = 0\, , \\
x_i &= -\frac{\lambda s}{d_i} \, , \\
s = \sum_i x_i &= - \lambda s \sum_{C \in \CC}\frac{1}{d_C}
\end{align*}

This implies that $s = 0$, which gives $\vx = 0$, or implies $1 = -\lambda \sum_i\frac{1}{d_i} < 0$, which is impossible.
Thus, $A^T A$ is a matrix of rank $|\CC|$, therefore $\rk A \geq |\CC|$, so $|\CC| \leq |E|$
\end{proof}

\begin{prop}[Generalised Fischer inequality]
Let $E$ be a finite set, $p$ a prime and $L \subseteq \F_p$.
A family $\FF \subseteq 2^E$ is said to be $L$-Fischer if for any $A, B \in \FF $ we have $|A\cap B| \mod p \in L$ whenever $A \neq B$ and, furthermore, $|A| \mod p \not \in L$ for all $A \in \FF$.

Then $|\FF| \leq \sum_{i=0}^{|L|} \binom{|E|}{i}$.
\end{prop}

\begin{proof}
We work in $\F_p^E$.
For each $F \in \FF$, let $\vv_F$ be the vector in $\F_p^E$ as above, and define the following polynomials in $\Z[\vx_e | e \in E]$:
$$f_F(\vx) \coloneqq \prod_{\ell \in L} (\vx \cdot \vv_F - \ell ) \, .$$

Note that for $A, B \in \FF$ distinct, we have $f_A(\vv_B) = 0$, whereas $f_A(\vv_A) \neq 0$.
We now rewrite each $f_F$ by replacing any monomial of the form $\prod_{e\in E}\vx_e^{\alpha_e}$ with $\prod_{\substack{e\in E\\ \alpha_E > 0}}\vx_e$.
It is still the case that for $A, B \in \FF$ distinct, we have $f_A(\vv_B) = 0$, whereas $f_A(\vv_A) \neq 0$.

We claim that $\{ f_F \}_{F \in \FF}$ is a linearly independent set.
Furthermore, because $\deg f \leq s$, each polynomial $f_F$ is in the vector space $\spn\{\prod_{e \in A} \vx_e | \, A \subseteq E, \, |A| \leq |L| \}$, so the theorem follows.

Indeed, if $\sum_{F \in \FF} f_F \alpha_F =0$, then evaluating this zero polynomial at $\vx = \vv_F $ for each $F \in \FF$, gives us that $\alpha_F f_F(\vv_F) = 0$, so $\alpha_F = 0$. 
We conclude that $\{ f_F \}_{F \in \FF}$ is a linearly independent set.
\end{proof}

The following was shown in \cite{hsieh1975intersection}:

\begin{prop}[Frankl-Wilson theorem]
Let $E$ be a finite set and $L \subseteq \Z$.
A family $\FF \subseteq 2^E$ is said to be $L$-Fischer if any two distinct $A, B \in \FF $ have $|A\cap B| \in L$.

Then $|\FF| \leq \sum_{i=0}^{|L|} \binom{|E|}{i}$.
\end{prop}

Note that the meaning of $L$-Fischer is intrinsically different whenever $L \subseteq \Z$ and $L \subseteq \F_p$.
We hope that this flexibility of definition does not bring any ambiguity.
Remark that, this time, we do not require $|A|\not\in L$, unlike in the $p$-adic case.

\begin{proof}
We work in $\Z^E$.
Write $\FF = \{F_1, \ldots , F_k\}$ with $|F_1| \leq |F_2| \leq \ldots \leq |F_k|$.
For each $F \in \FF$, let $\vv_F$ be the vector in $\Z^E$ as above, and define for $i= 1, \ldots , k$ the following polynomials in $\Z[\vx_e | e \in E]$:
$$f_i(\vx) \coloneqq \prod_{ \substack{\ell \in L \\ \ell < |F_i|}} (\vx \cdot \vv_{F_i} - \ell) \, .$$

Note that for $i > j $ elements in $[k]$, we have $F_i \not\subseteq F_j$, thus $|F_j \cap F_i| <  |F_i|$.
So we have $f_i(\vv_{F_j}) = 0$, whereas $f_i(\vv_{F_i}) \neq 0$.
We now rewrite each $f_j$ by replacing any monomial of the form $\prod_{e\in E}\vx_e^{\alpha_e}$ with $\prod_{\substack{e\in E\\ \alpha_E > 0}}\vx_e$.
It can be observed that it is still the case that for $i > j $ elements in $[k]$, we have $f_i(\vv_{F_j}) = 0$, whereas $f_i(\vv_{F_i}) \neq 0$.

We claim that $\{ f_F \}_{F \in \FF}$ is a linearly independent set.
Furthermore, because $\deg f \leq s$, each polynomial $f_F$ is in the vector space  $\spn\{\prod_{e \in A} \vx_e | \, A \subseteq E, \, |A| \leq |L| \}$, which has dimension $\sum_{i=0}^{|L|} \binom{|E|}{i}$, so the theorem follows from the independence claim.

Indeed, if $\sum_{i = 1}^k f_{F_i} \alpha_i =0$, let $j$ be the smallest index such that $\alpha_j \neq 0$.
But then evaluating this zero polynomial at $\vx = \vv_{F_j} $ gives us that $\alpha_{F_j} f_j(\vv_{F_j}) = 0$, so $\alpha_F = 0$. 
We conclude that $\{ f_F \}_{F \in \FF}$ is a linearly independent set.
\end{proof}

The following was shown in \cite{ray1975t}:

\begin{prop}[Ray-Chaudhuri-Wilson theorem]\label{prop:RCW}
Let $E$ be a finite set, $\lambda $ be a positive integer and $L$ a set of positive integers smaller than $\lambda$.
Assume that $\FF $ is an $L$-Fischer collection of subsets of $E$, with $|F| = \lambda $ for each $F \in \FF$.
Then $|\FF| \leq \binom{|E|}{|L|}$.
\end{prop}

\begin{proof}
We work in $\Z^E$.
For each $I \subseteq E$, let $\vv_I$ be the vector in $\Z^E$ as above, and define for $F \in \FF$ the following polynomials in $\Z[\vx_e | e \in E]$:
$$f_F(\vx) \coloneqq \prod_{\ell \in L } (\vx \cdot \vv_{F} - \ell ) \, .$$

Note that for $F, G \in \FF$ distinct, we have $f_F(\vv_G) = 0$, whereas $f_F(\vv_F) = \prod_{\ell \in L } (\lambda - \ell )\neq 0$.
We now rewrite each $f_F$ by replacing any monomial of the form $\prod_{e\in E}\vx_e^{\alpha_e}$ with $\prod_{\substack{e\in E\\ \alpha_E > 0}}\vx_e$.
It can be observed that for $F, G \in \FF$, we still have $f_F(\vv_G) = 0$ if and only if $F \neq G$.

Furthermore, for $I\subseteq E$, define 
$$g_I(\vx) \coloneqq (\lambda - \sum_{i\in E} \vx_i)\prod_{i\in I} \vx_i\, .$$

Note that for sets $I, J$ in $E$, we have that $g_I(\vv_J ) = 0$ whenever $I\not\subseteq J$.
Furthermore, we have $g_I(\vv_F) = 0$ for any $F \in \FF$.
We now rewrite each $g_I$ by replacing any monomial of the form $\prod_{e\in E}x_e^{\alpha_e}$ with $\prod_{\substack{e\in E\\ \alpha_E > 0}}x_e$.
It is still the case that for sets $I, J$ in $E$, we have that $g_I(\vv_J ) = 0$ whenever $I\not\subseteq J$, as well as that $g_I(\vv_F) = 0$ for any $F \in \FF$.

We claim that $\{ f_F \}_{F \in \FF}\cup \{g_I \}_{I \, : |I| < |L| }$ is a linearly independent set.
Furthermore, because $\deg f_F \leq |L|$ for $F \in \FF$, and $\deg g_I \leq |L|$ for $|I| < |L|$, the polynomials $f_F$ and $g_I$ are in the vector space $\spn\{\prod_{e \in A} \vx_e | \, A \subseteq E, \, |A| \leq |L| \}$, which as dimension $\sum_{i=0}^{|L|} \binom{|E|}{i}$.
It follows that
$$ \sum_{i=0}^{|L|} \binom{|E|}{i} \geq \sum_{i=0}^{|L|-1} \binom{|E|}{i} + |\FF| \, ,$$
so $\binom{|E|}{|L|} \geq |\FF|$ follows from the independence claim.

Indeed, assume that 
\begin{equation}\label{eq:l_i_ness}
\sum_{F \in \FF} f_F \alpha_F + \sum_{\substack{I \subseteq E\\ |I| < |L|}} g_I \alpha_I =0 \, .
\end{equation}

For any $F \in \FF$, evaluating \eqref{eq:l_i_ness} at $\vx = \vv_F $ gives us that $\alpha_F f_F(\vv_F) = 0$, so $\alpha_F = 0$.

If there is some $I$ with $|I| < |L|$ and $\alpha_I \neq 0$, find such $I$ minimal by inclusion.
Then evaluating \eqref{eq:l_i_ness} in $\vv_I$, using that $\alpha_F = 0$ for any $F \in \FF$, gives us $\alpha_I g_I(\vv_I) = 0$, which implies $\alpha_I = 0$.
We conclude that $\{ f_F \}_{F \in \FF}\cup \{g_I \}_{I \, : |I| < |L| }$ is a linearly independent set.
\end{proof}

\begin{prop}[Frankl - Wilson]\label{prop:frankl_wilson_fixed_size}
Let $E$ be a finite set, $p$ a prime number, $\lambda$ a positive integer and $L$ a collection of elements in $\F_p$.
Assume that $\FF$ is $L$-Fischer and $|F| = \lambda $ for each $F \in \FF$.

If $\lambda \mod p\not\in L$ and that $|L| \leq \mod \lambda$, then $|\FF| \leq \binom{|E|}{|L|}$.
\end{prop}

\begin{proof}
We work in $\F_p^E$.
For each $I \subseteq E$, let $\vv_I$ be the vector in $\F_p^E$ as above, and define for $F \in \FF$ the following polynomials in $\Z[\vx_e | e \in E]$:
$$f_F(\vx) \coloneqq \prod_{\ell \in L } (\vx \cdot \vv_{F} - \ell ) \, .$$

Note that for $F, G \in \FF$ distinct, we have $f_F(\vv_G) = 0$, whereas $f_F(\vv_F) = \prod_{\ell \in L } (\lambda - \ell )\neq 0$.
We now rewrite each $f_F$ by replacing any monomial of the form $\prod_{e\in E}\vx_e^{\alpha_e}$ with $\prod_{\substack{e\in E\\ \alpha_E > 0}}\vx_e$.
It can be observed that for $F, G \in \FF$, we still have $f_F(\vv_G) = 0$ if and only if $F \neq G$.

Furthermore, for $I\subseteq E$, define 
$$g_I(\vx) \coloneqq (\lambda - \sum_{i\in E} \vx_i)\prod_{i\in I} \vx_i\, .$$

Note that for sets $I, J$ in $E$, we have that $g_I(\vv_J ) = 0$ whenever $I\not\subseteq J$.
Note that $g_I(\vv_I) \neq 0$ whenever $|I| < |L|$, thanks to the condition $|L| \leq \lambda < p$.

Furthermore, we have $g_I(\vv_F) = 0$ for any $F \in \FF$.
We now rewrite each $g_I$ by replacing any monomial of the form $\prod_{e\in E}x_e^{\alpha_e}$ with $\prod_{\substack{e\in E\\ \alpha_E > 0}}x_e$.
It is still the case that for sets $I, J$ in $E$, we have that $g_I(\vv_J ) = 0$ whenever $I\not\subseteq J$, that $g_I(\vv_I) \neq 0$ for $|I| < |L|$ as well as that $g_I(\vv_F) = 0$ for any $F \in \FF$.

We claim that $\{ f_F \}_{F \in \FF}\cup \{g_I \}_{I \, : |I| < |L| }$ is a linearly independent set.
Furthermore, because $\deg f_F \leq |L|$ for $F \in \FF$, and $\deg g_I \leq |L|$ for $|I| < |L|$, the polynomials $f_F$ and $g_I$ are in the vector space $\spn\{\prod_{e \in A} \vx_e | \, A \subseteq E, \, |A| \leq |L| \}$, which as dimension $\sum_{i=0}^{|L|} \binom{|E|}{i}$.
It follows that
$$ \sum_{i=0}^{|L|} \binom{|E|}{i} \geq \sum_{i=0}^{|L|-1} \binom{|E|}{i} + |\FF| \, ,$$
so $\binom{|E|}{|L|} \geq |\FF|$ follows from the independence claim, which is done exactly as above.
\end{proof}

Here is a simple corollary from the proof above.

\begin{prop}[Frankl - Wilson generalisation]\label{prop:FWgen}
Let $E$ be a finite set, $p$ a prime number, $\lambda$ a positive integer and $L$ a collection of elements in $\F_p$.
Assume that $\FF$ is $L$-Fischer and $|F| = \lambda $ for each $F \in \FF$.

If $\lambda \mod p\not\in L$ and that $|L| \leq \mod \lambda$, then $|\FF| \leq \binom{|E|}{|L|} + \sum_{\substack{i < |L| \\ i = \lambda \mod p}} \binom{|E|}{i}$.
\end{prop}

\begin{prop}[Erd\"os - Ko - Rado]\label{prop:EKR}
Let $E$ be a finite set and  $\lambda $ a non-negative integer such that $2\lambda \leq |E|$.
Consider $\FF$  a family of sets in $E$, such that $|F| = \lambda $ for all $F \in \FF$.
Assume further that $F\cap G \neq \emptyset $ for any $F, G \in \FF$.

Then $|\FF| \leq \binom{n-1}{\lambda - 1}$.
\end{prop}

This proposition will be proven using spectral graph theory, presented below.

\section{Graph theory}

\begin{defin}[Graphs]
A graph $G = (V, E) $ is a pair of two sets, a vertex set $V$, and an edge set $E$.
An edge $e \in E$ is a set of two vertices $\{v, w\}$.
These vertices may be the same (in which case $e$ is called a \textbf{loop}), and there may be several copies of the same edge in $E$ (which are called \textbf{parallel edges}).
A graph with no loops or parallel edges is called a \textbf{simple graph}.
Unless otherwise stated, all our graphs will be simple.
\end{defin}

\subsection{Ramsey theory}
The central problem in \textbf{Ramsey theory} is computing the minimal $R \coloneqq R(m, n)$ such that any edge colouring of $K_R$ into two colours, say red and blue, contains either a monochromatic red $K_m$ or a monochromatic blue $K_n$.

\begin{defin}[Ramsey edge colouring]
An edge colouring of a complete graph is said to be an $(m, n)$-Ramsey colouring if there are no monochromatic red $K_n$ and no monochromatic blue $K_m$.

the number $R(m, n)$ is the smallest value for $R$ such that no edge bicolouring of $K_R$ is $(m, n)$-Ramsey.
\end{defin}

\begin{figure}[h]
\includegraphics[scale=1]{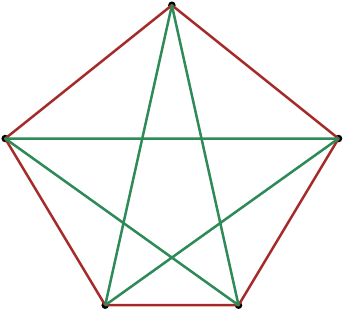}
\caption{An edge colouring of $K_5$ without any monochromatic triangle.\label{fig:k5_bicolour}}
\end{figure}

\begin{exe}\label{exe:r33}
Show that $R(3, 3) = 6$.
\cref{fig:k5_bicolour} may come in handy.
\end{exe}

The first few values obtained in \cite{chachamis2018ramsey} and in \cite{graver1968some} are shown in \cref{tab:ramsey_nums}.
It goes without saying that many of the values there presented are not as easy to compute as in the $R(3, 3)$ case from \cref{exe:r33}.
These numbers were obtained using vast amounts of ingenuity and computational power.

\begin{table}
\begin{tabular}{l|| c | c | c | c | c | c | c}
$R(m, n)$& $n = 1 $ & $n = 2 $ &$n = 3 $ &$n = 4 $ &$n = 5 $ & $n = 6$\\
\hline
$ m = 1 $& 1 & 1 & 1 & 1 & 1 & 1\\
$ m = 2 $& 1 & 2 & 3 & 4 & 5 & 6\\
$ m = 3 $& 1 & 3 & 6 & 9 & 14 & 18\\
$ m = 4 $& 1 & 4 & 9 & 18 & 25 & ??\\
$ m = 5 $& 1 & 5 & 14 & 25 & ?? & ??\\
$ m = 6 $& 1 & 6 & 18 & ?? & ?? & ??
\end{tabular}
\caption{Some of the known Ramsey numbers\label{tab:ramsey_nums}}
\end{table}

\begin{quote}
Suppose aliens invade the earth and threaten to obliterate it in a year's time unless human beings can find the Ramsey number for red five and blue five. We could marshal the world's best minds and fastest computers, and within a year we could probably calculate the value. If the aliens demanded the Ramsey number for red six and blue six, however, we would have no choice but to launch a preemptive attack.
\hfill (Paul Er\"os)
\end{quote}

\subsubsection{Probabilistic method in Ramsey bounds}
\begin{thm}
$$2^{n/_2} \leq R(n, n) \leq 2^{2n}\, . $$

\end{thm}

\begin{proof}
For the upper bound we use the following claim: for $m, n \geq 2$ we have.
\begin{equation}\label{eq:induction_ramsey}
R(n, m) \leq R(n, m- 1) + R(n-1, m) \, .
\end{equation}
That this is true for $n = 2$ and for $m = 2$ follows from the fact that $R(1, n) = 1$ and $R(2, n) = n$.

Now consider $N = R(n, m- 1) + R(n-1, m)$, and fix an edge bicolouring of $K_N$ into, say, red and blue colours.
Our goal is to find a monochromatic blue $K_n $ or a monochromatic red $K_m$ in $K_N$.

Pick one vertex $\vv$ and partition the remaning $N-1$ vertices $\vw$ into the blue set $V_b$ and the red set $V_r$ according to the edge of $\vv \-- \vw$.
Note that $|V_b| + |V_r| = N-1$, so either $|V_b| \geq  R(n-1, m)$ or $|V_r| \geq  R(n, m- 1)$.
In either case we have found either a blue $K_n$ or a red $K_m$.

As a conclusion, from \eqref{eq:induction_ramsey} we get inductively that $R(m, n) \leq 2^{m+n}$, thus $R(n, n) \leq 2^{2n}$.

For the lower bound, we observe in \cref{tab:ramsey_nums} that $R(n, n) > 2^{n/_2}$ for $n < 4$. 
Assume now that $n\geq 4$ and let $N = \lfloor 2^{n/_2} \rfloor $ and colour each edge of $K_N$ at random uniformly.
Let $R_I$ be the event that the vertex set $I$ induces a monochromatic red complete graph, and let $B_I$ be the event that the vertex set $I$ induces a monochromatic blue complete graph.
Let $R$ be the event that some set $I$ induces a monochromatic red complete graph, and let $B$ be the event that some set $I$ induces a monochromatic blue complete graph.
Use $\overline{A}$ to denote the negation of an event $A$.
We have that:
\begin{align*}
\PP[\overline{B} \cap \overline{R}] =& \PP[\bigcap_{|I| = n} \overline{B_I} \cap \bigcap_{|I| = n} \overline{R_I}]\\
=& 1 - \PP[\bigcup_{|I| = n} B_I \cup \bigcup_{|I| = n} R_I]\\
\geq & 1 - \sum_{|I| = n} \PP[B_I] + \PP[R_I]\\
=& 1 - 2\binom{N}{n} 2^{-\binom{n}{2}}\, .
\end{align*}
Now observe that $\binom{N}{n} \leq\frac{N^n}{n!}$ and $N \leq 2^{n/_2}$, so
$$2\binom{N}{n} 2^{-\binom{n}{2}} \leq \frac{2 N^n}{n! N^{n-1}} = \frac{2N}{n!} \leq \frac{ 2 2^{n/_2}}{n!} \leq 0.5 \, ,$$
where the last inequality is equivalent to $4 \cdot 2^{n/_2} \leq n!$ which  holds for $n \geq 4$.

Thus $\PP[\overline{B} \cap \overline{R}]  \geq 0.5$, so there is a random assignment of a bicolouring that makes $K_N$ an $(n, n)$-Ramsey colouring.
We conclude that $R(n, n) \geq 2^{n/_2}$.
\end{proof}

A key issue with the lower bound described above is that it does not provide an explicit construction of a bicolouring that has neither monochromatic red $K_n$ nor a monochromatic blue $K_m$.
It just says that, with a randomly generated bicolouring, you will most probably end up with an $(n, n)$-Ramsey colouring.
But do not forget, to test that a bicolouring is $(n, n)$-Ramsey colouring is hard.

For this reason, we dedicate the next section to present some expicit constructions of bicolourings using algebraic tools.

\subsubsection{Explicit constructions}

The following has no monochromatic $K_n$.

\begin{const}[Construction for $N = O(n^2)$ - naive construction]
The union of $n-1$ red complete graphs, with blue edges between these.
\end{const}

The following is a construction from \cite{nagy1972certain}, also presented in \cite{chung1981note}.

\begin{const}[Construction for $N = O(n^3)$ - Nagy's graph]
Let $N = \binom{n-1}{3}$ and identify the vertices of $K_N$ with subsets of $[n-1]$ of size three.
Colour $A \-- B$ blue if $|A \cap B|$ is odd, and colour it red if $|A\cap B|$ is even.
This bicolouring has no monochromatic $K_n$.
\end{const}

\begin{proof}
Indeed, if $\FF = \{A_1, \ldots, A_n\}$ is a monochromatic blue clique, this is a $1$-Fischer family on $E = [n-1]$, which is impossible according to \cref{prop:fischer}.

If $\FF = \{A_1, \ldots, A_n\}$ is a monochromatic red clique, this is an oddtown family on $E = [n-1]$, which is impossible according to \cref{prop:oddtown}.
\end{proof}

The following was presented in \cite{frankl1981intersection}.

\begin{const}[Construction for $N = O(n^k)$ for $k$ arbitrarily large]
Fix $p$ prime number, $n$ positive integer, and let $N=\binom{n}{p^2-1}$.
In $K_N$, we identify the vertices with subsets of $[n]$ of size $p^2-1$.
We colour $X \-- Y $ blue if $|X\cap Y| \neq 1 \mod p$, and we colour $X \-- Y $ red otherwise.
Then there is no $\binom{n}{p-1}+1$ monochromatic complete graph in $K_N$.
\end{const}

\begin{proof}
Assume that $\CC = \{C_1, \ldots, C_m\}$ is a monochromatic blue subgraph of $K_N$.
Then it is $L$-Fischer, where $L = \{0, 1, \ldots, p-2\} \subseteq \F_p$, with all subsets of size $p^2-1$.
We can therefore use \cref{prop:frankl_wilson_fixed_size}, which gives us that $m \leq \binom{n}{|L|} = \binom{n}{p-1}$.

Now assume that $\CC = \{ C_1, \ldots,  C_m\}$ is a monochromatic red subgraph of $K_N$.
Then it is $L$-Fischer, where $L = \{p-1, 2p-1, \ldots p(p-1) - 1\}$, where each set has constant size $p^2-1$.
Thus according to \cref{prop:RCW}, we have that $m \leq \binom{n}{|L|}$, this shows that the presented colouring is $(\binom{n}{p-1}+1, \binom{n}{p-1}+1)$-Ramsey, therefore $R(\binom{n}{p-1}+1, \binom{n}{p-1}+1) \geq \binom{n}{p^2-1}$.

For $p$ fixed and $n$ large, this shows that $R(n, n) \geq O(n^p)$.
\end{proof}

\subsubsection{The chromatic polynomial}

\begin{defin}[Chromatic polynomial]
Given a graph $G$, its \textbf{chromatic polynomial} $\chi_G$ is a map $\Z_{\geq 0} \to \Z_{\geq 0}$ given by
$$\chi_G(n) \coloneqq \#\{f:[n] \to V(G) | \text{ $f(v) \neq f(w)$ for neighbour vertices $v, w$}\}\, . $$
The functions counted by $\chi_G(n)$ are called \textbf{stable colourings} of $G$ with $n$ colours.
Note that in this context we are colouring the \textbf{vertices} of the graph.
\end{defin}

\begin{defin}[Deletion and contraction of edges]
If $G=(V, E)$ is a graph and $e$ one of its edges, the \textbf{deletion of $e$} is denoted by $G\setminus e$, and is the graph with the same vertex set, and with edge set $E\setminus e$.

If $G=(V, E)$ is a graph and $e=\{v, w\}$ one of its edges, the \textbf{contraction of $e$} is denoted by $G/_e$, and is the graph with vertex set $\{\star\} \cup V \setminus\{v, w\}$, resulting in a graph with one fewer vertex.
The edge set is in bijection to $E\setminus e$.
Any edge $e$ that is incident to either $v$ or $w$ becomes now incident to the new vertex $\star $.

Note that simple graphs may not remain simple graphs under this operation.
\end{defin}

\begin{figure}
\includegraphics[scale=.61]{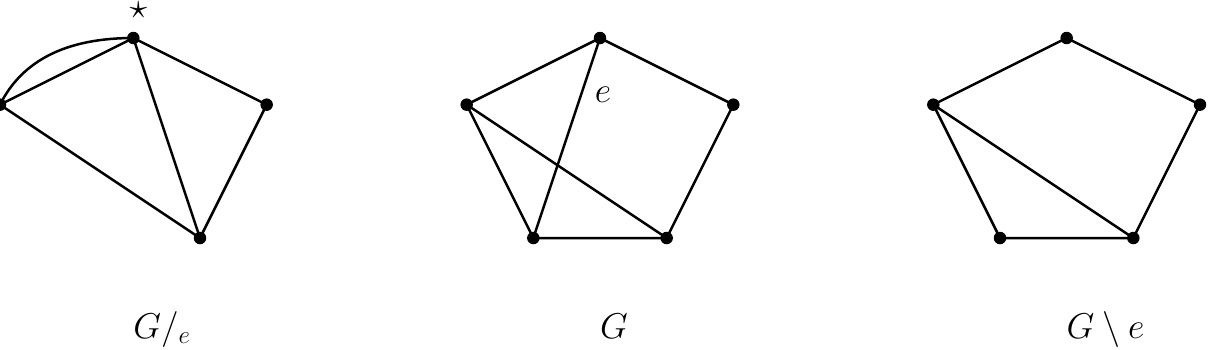}
\caption{The deletion and a contraction of an edge\label{fig:contraction}. The newly created vertex under the contraction operation is labelled with $\star$.}
\end{figure}

\begin{thm}
The chromatic polynomial of a graph is indeed a polynomial.
In fact, it satisfies a \textbf{deletion contraction formula}: for any edge $e$ of the graph $G$, 
$$\chi_G(n) = \chi_{G\setminus e}(n) - \chi_{G/_e}(n)\, .$$
\end{thm}

The \textbf{chromatic number} of a graph $G$ is the smallest integer $n$ such that $\chi_G(n) \neq 0$.

\begin{smpl}
Consider the Petersen graph $P$, a graph with $10$ vertices and $15$ edges displayed in \cref{fig:petersen}.

It is possible to colour this graph in a stable way using three colours, but not two, so $\chi(P) = 3$.
It can be computed that 
$$\chi_P(n) = n ( n-1) (n-2) (x^7 -12 x^6+67 x^5 - 230 x^4 + 529 x^3 - 814 x^2 + 775 x - 352 ) \, . $$
\end{smpl}

\begin{figure}
\includegraphics[scale=1]{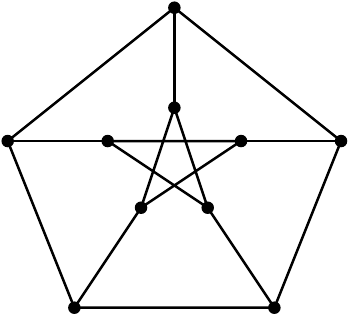}
\caption{The Petersen graph\label{fig:petersen}}
\end{figure}

\subsubsection{The chromatic number}

\begin{prop}
Let $G$ be a graph, and $\Delta $ the maximal degree of a vertex in $G$.
Then $\chi(G) \leq \Delta + 1$.
\end{prop}

\begin{proof}
Use a greedy algorithm to colour the graph: when you see an uncoloured vertex, use the next available colour.
As any vertex has at most $\Delta $ neighbours, there are always available colours when using $\Delta + 1$ colours.
\end{proof}

In $\R^d$ we construct a graph, by connecting $\vx, \vy$ if $||\vx - \vy|| = 1$.
In this section we give bounds for the chromatic number of this graph

\begin{thm}[Colouring the plane]
$$\chi (\R^2) \in \{4, 5, 6, 7\}\, .$$
\end{thm}

\begin{proof}
The picture in \cref{fig:plane_chromat} shows a graph that can be embedded in the plane, in such a way that the image of any two neighbouring vertices is at distance one of each other.

This graph cannot be coloured with three colours in a proper way, so the plane cannot be coloured with three colours in a proper way.

A proper colouring of the plane using seven colours is sketched in \cref{fig:plane_chromat}.


\begin{figure}[h]
\centering
\includegraphics[scale=0.7]{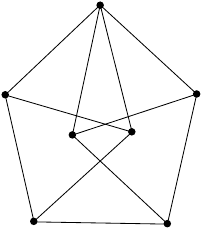}
\hspace{2cm}
\includegraphics[scale=0.3]{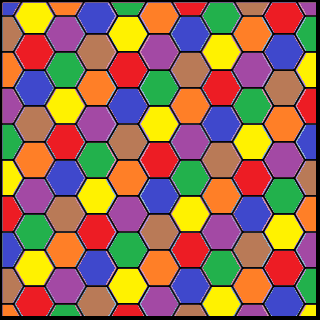}
\caption{\textbf{Left} A graph with chromatic number three, that can be isometrically embedded in the plane. \textbf{Right} A proper colouring of the plane using seven colours.\label{fig:plane_chromat}}
\end{figure}

\end{proof}

\begin{thm}
$$\chi(\R^d) \leq 9^d \, .$$
\end{thm}

\begin{proof}
This proof is not presented here
\end{proof}

\subsection{Spectral theory of graphs}

\subsubsection*{Preliminary observations}

The adjacency matrix of a graph $G$, denoted as $A_{G} = 
\begin{bmatrix}
a_{i,j}
\end{bmatrix}_{i, j}$ is a $V(G)\times V(G)$ matrix.
The incidence matrix of a graph $G$, denoted as $B_{G} = 
\begin{bmatrix}
b_{i,j}
\end{bmatrix}_{i, j}$, is a $V(G) \times E(G)$ matrix.
These are defined so that

$$ a_{i, j} =\begin{cases*}
      & 1 \text{ if $i \-- j$ in $G$,}\\
      & 0 \text{ otherwise.}
    \end{cases*}  \text{ and }
b_{i, e} =\begin{cases*}
      & 1 \text{ if $e$ is incident in $j$ in $G$,}\\
      & 0 \text{ otherwise.}
    \end{cases*}
    $$

\begin{smpl}
Recall that $K_3$ is the complete graph on three vertices, that is a triangle.
Then $A_{K_3} = 
\begin{bmatrix}
0 & 1 & 1\\
1 & 0 & 1\\
1 & 1 & 0
\end{bmatrix}$.

Let $G$ be the complete graph on four vertices with one edge removed.
In this way, it has four vertices and five edges.
Then $A_{G} = 
\begin{bmatrix}
0 & 1 & 1 & 0\\
1 & 0 & 1 & 1\\
1 & 1 & 0 & 1\\
0 & 1 & 1 & 0
\end{bmatrix}$.

\end{smpl}

The matrix $A_G$ is not an \textbf{invariant} of the graph $G$, as it depends on the specific order of the vertices that we pick to describe the matrix.
However, the \textbf{spectrum} does not depend on this order, so it is an invariant of $G$.

\begin{defin}[Spectrum of a graph]
Given a graph $G$, the \textbf{eigenvalues} of $A_G$, as a multiset, form the \textbf{spectrum} of $G$, denoted $\spec \, G$.
\end{defin}

A generic $n\times n$ matrix has eigenvalues that are complex numbers, possibly with some repetitions.
Because the matrix $A_G$ is symmetric, we know that its eigenvalues are real.
We also know that when there are repeated eigenvalues, these come with an increase in the dimension of the eigenspace.

\begin{thm}[Spectral theorem of symmetric matrices]\label{thm:spect}
Let $A$ be a symmetric matrix, with eigenvalues $\spec \, A = \{\lambda_i\}_i$ and corresponding eigenspaces $V_{\lambda} = \ker(A - \lambda \Id)$.

\begin{itemize}
\item All eigenvalues are real values.

\item All eigenspaces are complete, that is $\dim (V_{\lambda})$ is precisely the multiplicity of $\lambda $ in $\spec A$.

\item There is an orthogonal basis of eigenvectors $\{v_i\}_i$.
That is, we have $A \vv_i = \lambda_i \vv_i$ and $\vv_i \cdot \vv_j = 0 $ for $i \neq j$.
\end{itemize}
\end{thm}

In this way, for a graph $G$, we denote $\spec \, G \coloneqq \{ \lambda_1 \geq  \dots \geq \lambda_n \}$ for the multiset of eigenvalues, ordered in decreasing order.
We also write $\{ \vv_1, \ldots , \vv_n\}$ for the corresponding orthonormal basis, so that $A_G \vv_i = \lambda_i \vv_i$.

A bipartite graph $G$ is a graph whose vertices can be partitioned $V(G) = S\uplus T$ so that both $S, T$ are independent sets.
Denote de complement of a graph $G$ by $\overline{G}$.

\begin{lm}\label{lm:prelspec}
\begin{itemize}
The following facts help us compute the spectrum of graphs: 

\item If $G$ is $d$-regular, then $\lambda_1 = d$ and $\vv_1 = \frac{1}{\sqrt{n}}(1 \cdots 1)^T$

\item If $G$ is conected, then $\lambda_1 \neq \lambda_2$.

\item A graph $G$ is bipartite if and only if $-\spec \, G = \spec \, G$.

\item If $G$ is $d$-regular, denote $\spec \, \overline{G} = \{\overline{\lambda_{n+1}} \geq \overline{\lambda_{n}} \geq \cdots \geq \overline{\lambda_2}\}$, and the corresponding orthonormal eigenbasis by $\{ \overline{\vv_{n+1}}, \ldots, \overline{\vv_2} \}$.
Then 
$\overline{\lambda_{n+1}} = n - d - 1$ and 
$$\overline{\lambda_i} = -1 - \lambda_i, \text{ and }  \overline{\vv_i} = \vv_i \text{ for all $i = 2, \ldots , n$.}$$
\end{itemize}
\end{lm}

\begin{proof}
This proof is not presented here.
\end{proof}

\begin{prop}
Recall that, for a graph $G$, we write $L(G)$ for its line graph.
Then we have
$$B_G B_G^T = A_{L(G)} + 2\, \Id_{E(G)} \, .$$

Furthermore, if $G$ is a $d$-regular graph, then
$$B_G B_G^T = A_G + d\, \Id_{V(G)}\, . $$
\end{prop}

\begin{proof}
This proof is not presented here.
\end{proof}

\subsubsection*{Schwenks theorem}

A \textbf{cover} of a graph $G$ with a collection of graphs $(G_1, \ldots , G_k)$ is a collection of bijective maps $f_i: V(G_i) \to V(G)$ such that:
\begin{itemize}
\item If $v \-- w$ is an edge in $G_i$, then $f_i(v) \-- f_i(w)$ is an edge in $G$.

\item If $v \-- w $ is an edge in $G$, there is a unique $i$ such that $f_i^{-1}(v) \-- f_i^{-1}(w)$ is an edge in $G_i$.
\end{itemize}

\begin{lm}
The spectrum of the Petersen graph is 
$$ \{ 3, 1, 1, 1, 1, 1, -2, -2, -2, -2 \} \, .$$
\end{lm}

\begin{proof}

\end{proof}

\begin{thm}[Schwenk's theorem]\label{thm:schwenk}
There is no cover of $K_{10}$ with $(P, P, P)$, that is using three copies of the Petersen graph.
\end{thm}

\begin{proof}[Proof of \cref{thm:schwenk}]
If we have a cover of $K_{10}$ with three copies of $P$, then there are three matrices $A_1, A_2, A_3$ that arise from the adjadency matrix of $P$ by permuting rows and columns, that satisfy
$$A_1 + A_2 + A_3 = A_{K_{10}} = \mathbb{1}\mathbb{1}^T - \Id_{10}\, . $$

Observe that $\mathbb{1}$ is an eigenvector of $A_1, A_2, A_3$ because the Petersen graph is $d$-regular.

Because they arise by appplying a permutation to the rows and columns of the adjacency matrix of the Petersen graph, the three matrices $A_1, A_2, A_3 $ have the same spectrum.

Let $V^1_1$ be the eigenspace of $A_1$ corresponding to the eigenvector $1$.
Let $V^2_1$ be the eigenspace of $A_2$ corresponding to the eigenvector $1$.
These are five dimensional spaces orthogonal to $\mathbb{1}$, so their intersection has to be non-trivial.
Let $\vv \in V^1_1 \cap V^2_1$ be non-zero vector.
Then
\begin{align*}
& A_1 \vv + A_2 \vv + A_3 \vv  = (\mathbb{1}\mathbb{1}^T - \Id_{10}) \vv \\
\Leftrightarrow & 2 \vv + A_3 \vv = - \vv \\
\Leftrightarrow & A_3 \vv = - 3 \vv \, ,
\end{align*}
which is impossible, because $-3$ is not in the spectrum of $A_3$.
\end{proof}

\subsubsection*{Hoffman bound and spectral gaps}

\begin{lm}\label{lm:spectral_ineq}
Let $G$ be a $d$-regular graph on $n$ vertices, with spectrum $\{ \lambda_1 \leq \ldots \leq \lambda_n\}$.
Then we have
\begin{itemize}
\item $\sum_{i \-- j} (\vx_i - \vx_j)^2 \leq (d - \lambda_n)\sum_i x_i^2$.

\item If $\sum_i \vx_i = 0$ then $\sum_{i \-- j} (\vx_i - \vx_j)^2 \geq (d - \lambda_2)\sum_i x_i^2$.
\end{itemize}
\end{lm}

\begin{proof}
We start with a rearrangement:
\begin{align*}
\sum_{i \-- j} (\vx_i - \vx_j)^2 =& \sum_{i \-- j}\left( \vx_i^2 - 2 \vx_i \vx_j + \vx_j^2\right) \\
=& \sum_{i= 1}^n d \vx_i^2 - 2\sum_{i\-- j} \vx_i \vx_j \\
=& d \vx \vx^T - \vx A_G \vx^T\, . 
\end{align*}

Recall that we write $\vv_1, \ldots , \vv_n$ for the orthonormal eigenbasis of $A_G$.
Write $\vx = \sum_{i=1}^n \alpha_i \vv_i$.
In this way, we have
$$\vx \vx^T = \sum_{i=1}^n \sum_{j=1}^n \alpha_i \alpha_j \vv_i \vv_j^T = \sum_{i=1}^n \alpha_i^2\, , $$
because $\vv_i \vv_j^T = 0$ for $i \neq j$, and $\vv_i \vv_i^T = ||\vv_i||^2 = 1$.

Using the same facts, we have
\begin{align*}
\vx A_G\vx^T =& \sum_{i=1}^n \sum_{j=1}^n \alpha_i \alpha_j \vv_i A_G\vv_j^T\\
=& \sum_{i=1}^n \sum_{j=1}^n \alpha_i \alpha_j \lambda_j \vv_i \vv_j^T \\
=&\sum_{i=1}^n  \alpha_i^2 \lambda_i\, . 
\end{align*}

We conclude that
$$ \sum_{i \-- j} (\vx_i - \vx_j)^2 = \vx \vx^T - \vx A_G \vx^T = \sum_{i=1}^n  \alpha_i^2 (d  - \lambda_i) \, . $$

The first item follows from $d  - \lambda_i \leq d  - \lambda_n$, along with $\vx \vx^T = \sum_{i=1}^n \alpha_i^2$.
For the second item we are given that $\vx (1 \cdots 1)^T = \sum_{i=1}^n \vx_i = 0$, but \cref{lm:prelspec} predicts that $\vv_1 = \frac{1}{\sqrt{n}} (1 \cdots 1)^T$, so
$$0 = \vx \vv_1^T = \sum_{i=1}^n \alpha_i \vv_i \vv_1^T = \alpha_1 \, .$$

We conclude with 
$$ \sum_{i \-- j} (\vx_i - \vx_j)^2 = \sum_{i=1}^n  \alpha_i^2 (d  - \lambda_i) = \sum_{i=2}^n  \alpha_i^2 (d  - \lambda_i) \geq \sum_{i=2}^n  \alpha_i^2 (d  - \lambda_2) = \vx \vx^T(d-\lambda_2) \, . $$
\end{proof}

\begin{prop}[Hoffman bound]\label{prop:indep_bound}
Let $G$ be a $d$-regular graph with lowest eigenvalue $\lambda_n$.
Then, the independence number $\alpha(G)$ satisfies
$$\alpha(G) \leq n \frac{-\lambda_n}{d - \lambda_n} = \frac{n}{\frac{d}{-\lambda_n} + 1}\, . $$
\end{prop}

\begin{proof}
Let $A$ be an independent set in $G$ of size $a$.
Define $\vx \in \R^{V(G)}$ via 
$$ \vx_i =\begin{cases*}
      & n-a \text{ if $i \in A$,}\\
      & -a \text{ otherwise.}
    \end{cases*} $$

In this way, $\sum_i \vx_i  = 0$.
Furthermore, $\sum_i \vx_i^2 = a(n-a)^2 + (n-a)a^2 = n a (n-a)$.
On the other hand, because $G$ is $d$-regular,
$$ \sum_{i \-- j} (\vx_i - \vx_j)^2 = \sum_{i \in A}\sum_{j \not\in A}(\vx_i - \vx_j)^2 = dan^2\, . $$

Applying this to \cref{lm:spectral_ineq} gives us $dan^2 \leq (d-\lambda_n)an (n-a)$, which can be rearranged to $a \leq n \frac{-\lambda_n}{d - \lambda_n}$.
\end{proof}

\begin{proof}[Proof of \cref{prop:EKR}]
Let $G$ be a graph with $\binom{|E|}{\lambda}$ vertices, which we identify with the subsets of $E$ of size $\lambda$.
We draw an edge $A \-- B$ if $A \cap B = \emptyset $.
The family $\FF $ as described is an independent set in the graph $G$, therefore $|\FF | \leq \alpha (G)$.

On the other hand, $G$ is $d$-regular, with $d = \binom{n - \lambda}{ \lambda}$.
We now see that $\lambda_n$, the smallest eigenvalue of the spectrum of $G$, is $-\binom{|E| - \lambda - 1}{\lambda - 1}$.

It follows from \cref{prop:indep_bound} that $\alpha(G) \leq |E|\frac{-\lambda_n}{d - \lambda_n} = |E|\frac{\binom{|E| - \lambda - 1}{\lambda - 1}}{\binom{n - \lambda}{\lambda} + \binom{|E| - \lambda - 1}{\lambda - 1}} = \binom{|E|-1}{\lambda - 1}$.
\end{proof}

\subsubsection*{Erd\"os-Ko-Rado}

Recall from above the Erd\"os-Ko-Rado theorem:

\begin{thm}\label{thm:EKR}
Let $E$ be a finite set, and $\mathcal F \subseteq 2^E$, such that for any $A, B \in \FF$ we have $A \cap B \neq \emptyset $.

Then $|\FF | \leq \binom{|E|-1}{r-1}$.
Furthermore, the $\FF$ that attain equality are precisely the ones given by $\FF_i \coloneqq \{ V \in \binom{E}{r} | i \in V\}$, where $i \in E$.
\end{thm}

Note that, if $n < 2r$ the theorem is straightforward, as $\binom{|E|}{r} \leq \binom{|E|-1}{r-1}$.
We will therefore consider only the case where $n \geq 2r$.

For that, we construct the \textbf{Kresner graph}.

\begin{defin}[Kneser graph]
Given a set $E$ and an integer $r$, the Kneser graph $K(E, r)$ has vertex set $\binom{E}{r} = \{ V \subseteq E : |V| = r\}$, and an edge $V \-- W$ if $V \cap W = \emptyset$.
\end{defin}

\begin{smpl}
If we take $E = \{1, 2, 3, 4, 5\}$ and $r = 3$, 
\end{smpl}

\begin{lm}
The spectrum of $K(E, r)$ is 
$$ \left\{(-1)^i\binom{|E| -r - i}{r - i}^{\times \binom{|E|}{i} - \binom{|E|}{i-1}} \Big| i = 0, 1 \ldots , r \right\} \, . $$
\end{lm}

\begin{proof}
This proof is not presented here.
\end{proof}

\begin{proof}[Proof of \cref{thm:EKR}]
Let $G = K(E, r)$.
The set $\FF$ is an independent set in $G$, so it is enough to show that $\alpha(G) \leq \binom{|E|-1}{r-1}$.

Let $n = |V(G)| = \binom{|E|}{r}$.
Note that  $G$ is $d$-regular with $d = \binom{|E|-r}{r}$ and $\lambda_n = -\binom{|E| -r -1 }{r - 1}$.
By the Hoffman bound, we have
$$ \alpha(G) \leq -\frac{n\lambda_n}{d - \lambda_n} = \binom{|E|-1}{r-1} \, ,$$
as desided.
\end{proof}

\subsubsection*{Max cut bounds}

\begin{defin}[Max cut of a graph]
A \textbf{cut} of a graph $G$ is a partition $V = R \uplus L$.
The \textbf{weight} of a cut $V = R \uplus L$ is the number of edges that cross the cut, that is edges $e = a \-- b$ such that $a \in R$ and $b \in L$, denoted $\we(R, L)$.

The \textbf{max cut} of a graph is defined as
$$ \maxcut (G) \coloneqq \max_{V = R \uplus L} \we(R, L) \, . $$
\end{defin}

\begin{prop}
$$\maxcut(G) \geq \frac{1}{2}|E| \, .$$
\end{prop}

\begin{proof}
Application of the probabilistic method.
Proof not given here.
\end{proof}

\begin{prop}
Let $G$ be $d$-regular graph on $n$ vertices, and let $\lambda_n$ be its largest eigenvalue.
Then 
$$\maxcut(G) \leq \frac{1}{2} |E| - \frac{n \lambda_n}{4} = \frac{n (d-\lambda_n)}{4} \, . $$
\end{prop}

\begin{proof}
This proof is not presented here.
\end{proof}

\subsubsection*{Friendship graphs}

\begin{defin}[Windmills and friendship graphs]
A graph $G$ is said to be a \textbf{windmill} if there is a special vertex $v$, called its \textbf{centre}, that shares an edge with all other vertices and the graph resulting from removing $v$ is a perfect matching.

A graph $G$ is said to be a \textbf{friendship} graph if any two distinct vertices have a unique common neighbour.
\end{defin}

\begin{figure}[h]
\includegraphics[scale=.65]{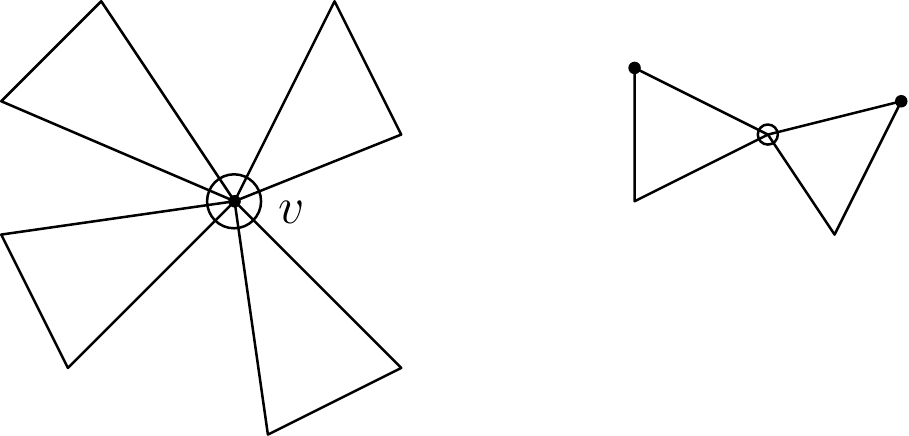}
\caption{\textbf{Left:} A windmill graph, with its centre highlighted. \textbf{Right:} A friendship graph, with a particular common neighbour highlighted.\label{fig:windmill}}
\end{figure}

It is clear that any windmill is a friendship graph.
The following result was established in 1966 and proves the converse.

\begin{thm}[Erd\"os-R\'enyi-Sos theorem]\label{thm:ERS}
Every friendship graph is a windmill.
\end{thm}

\begin{lm}\label{lm:FWd}
If $G$ is a friendship graph, it is either $d$-regular, or it is a windmill.
\end{lm}

\begin{proof}
Not given here.
\end{proof}

\begin{proof}[Proof of \cref{thm:ERS}]
Assume otherwise, and let $G$ be a friendship graph that is not a windmill.
We can assume that $G$ is $d$-regular, from \cref{lm:FWd}.

Then, 
$A_{G}^2 = 
\begin{bmatrix}
d & 1 & \cdots & 1 & 1\\
1 & d & & 1 & 1\\
\vdots & & \ddots &  & \vdots\\
1 & 1 & & d & 1\\
1 & 1 & \cdots & 1 & d
\end{bmatrix}$.

The spectrum of this matrix can be directly computed as $\{n + d - 1, d-1^{\times n-1}\}$.
It is easy to see that $\spec \, A^2 = \{ \lambda^2 | \lambda \in \spec \, A\}$.
Because $G$ is $d$-regular, $d$ is the largest eigenvalue of $A_G$, and so either $d^2 = n+d-1$ or $d^2 = d-1$.

The first one gives us $n = d^2 - d + 1$ and the latter one is impossible for $d$ integer.
This gives us that the spectrum of $G$ is $\{d, \sqrt{d-1}^{\times a}, -\sqrt{d-1}^{\times b}\}$, where $a, b$ are non-negative integers.

Because $\tr(A_G ) = 0$, we have that $ d + a\sqrt{d-1} - b\sqrt{d-1} = 0$, or $b-a = \frac{d}{\sqrt{d-1}}$.

This means that $\sqrt{d-1} = x$ is an integer that divides $d^2$, but $d^2 = x^2 + 1$ so $x $ divides $1$.
We conclude that either $x = 1$ or $x = -1$.

That is either $d = 2, n = 3$ or $d = 1, n=1$.
These give us the graphs $K_1$ and $K_3$, which also are windmills.
We conclude that any friendship graph is a windmill.
\end{proof}

\subsubsection*{Sensitivity conjecture}

\begin{thm}\label{thm:zhang}
Let $W$ be a subset of $V(C_n)$ of size at least $2^{n-1}+1$.
Then $\Delta(C_n|_W) \geq \sqrt{n}$.
\end{thm}

\begin{smpl}
We observe that there is some $W \subset V(C_n)$ with $2^{n-1}$ independent elements.

Specifically, if we take $W = \{ \vv \in {0, 1}^n | \sum_i \vv_i \text{ is even }\}$, this set has $2^{n-1}$ elements, none of which share an edge any edge.
So $\Delta(C_n|_W) = 0$.
\end{smpl}

To establish the proof of \cref{thm:zhang}, we need some auxiliary lemmas.
The first one reinterprets the eigenvalues of a matrix as extremals values of its corresponding quadratic form:

\begin{lm}[Maximal quadratic form]\label{lm:maximalquad_EV}
Let $A$ be an $n \times n$ symmetric matrix with spectrum $\{\lambda_1 \geq \cdots \geq \lambda_n\}$.
Then we have the following formulas, for each $j = 1, \ldots , n$:
$$\lambda_j = \max_{\substack{U \subseteq \R^n \\ \dim U = j}} \min_{\vx \in U} \frac{\vx^T A \vx}{\vx^T\vx} = \min_{\substack{U \subseteq \R^n \\ \dim U = n-j+1}} \max_{\vx \in U} \frac{\vx^T A \vx}{\vx^T\vx} \, ,$$
where the are taking maxima and minima over vector spaces $U \subseteq \R^n$ with a fixed dimension.
\end{lm}

\begin{proof}
Let $\{\vv_1, \cdots , \vv_n\}$ be an orthonormal eigenbasis of $A$, which we see it exists in \cref{thm:spect}.
Write $\vx$ is this basis, arising the coefficients $\{\alpha_i\}_{i=1}^n$ such that:
$$ \sum_i \alpha_i \vv_i = \vx \, . $$

Note that $\langle \vx \vv_k \rangle = \sum_{i=1}^n \alpha_i \langle \vv_i, \vv_k \rangle = \alpha_k$.

We have that 
\begin{align*}
\vx^T A \vx &= \sum_{i=1}^n \sum_{j=1}^n \alpha_i \alpha_j \vv_i^T A \vv_j \\
&= \sum_{i=1}^n \sum_{j=1}^n \alpha_i \alpha_j \lambda_j \vv_i^T \vv_j\\
&= \sum_{i=1}^n \alpha_i^2 \lambda_i\, , 
\end{align*}
and that 
$$ \vx^T \vx = \sum_{i=1}^n \sum_{j=1}^n \alpha_i \alpha_j \vv_i^T \vv_j = \sum_{i=1}^n \alpha_i^2\, . $$

For a vector space $U$, define $M(U) = \max_{\vx \in U} \frac{\vx^T A \vx}{\vx^T\vx}$ and $m(U) = \min_{\vx \in U} \frac{\vx^T A \vx}{\vx^T\vx}$.
Let as well $M = \max_{\substack{U \subseteq \R^n \\ \dim U = j}} m(U)$ and $m = \min_{\substack{U \subseteq \R^n \\ \dim U = n-j+1}} M(U)$.
Our goal is to show that $\lambda_j = M = m$.
We will split the proof into four parts: $\lambda_j \leq M$,  $\lambda_j \geq M$,  $\lambda_j \leq m$ and $\lambda_j \geq m$.

To show that  $\lambda_j \leq M$, we find a vector space $U$ such that $\dim U = j $ and $m(U) \geq \lambda_j$.
Indeed, let $U = \spn \{\vv_i\}_{i=1}^j$.
It is imediate that $\dim U = j$ and since any $\vx \in U $ can be written as $\vx = \sum_{i=1}^j \alpha_i \vv_i$, we have that 
$$\frac{\vx^T A \vx}{\vx^T \vx} = \frac{\sum_{i=1}^j \alpha_i^2 \lambda_i}{\sum_{i=1}^j \alpha_i^2} \geq \frac{\sum_{i=1}^j \alpha_i^2 \lambda_j}{\sum_{i=1}^j \alpha_i^2}= \lambda_j\, . $$

To show that $\lambda_j \geq M$, we show that any vector space $U$ with $\dim U = j $ has some vector $\vx \in U $ such that $\frac{\vx^T A \vx}{\vx^T \vx} \leq \lambda_j$.
Indeed, the space $V = \{ \vx \in \R^n | \langle \vx , \vv_i \rangle = 0 \text{ for $ i = 1, \ldots, j-1$ } \}$ has codimension $j-1$, so the intersection $U\cap V$ is non-trivial.
Let $\vx \in U \cap V$, then $\alpha_i = \langle \vx , \vv_i \rangle = 0 $ for $i = 1, \ldots, j-1$, thus 
$$\frac{\vx^T A \vx}{\vx^T \vx} = \frac{\sum_{i=j}^n \alpha_i^2 \lambda_i}{\sum_{i=j}^n \alpha_i^2} \leq \frac{\sum_{i=j}^n \alpha_i^2 \lambda_j}{\sum_{i=j}^n \alpha_i^2}= \lambda_j\, . $$

To show that $\lambda_j \leq m$, we show that any vector space $U$ with $\dim U =  n - j + 1 $ has some vector $\vx \in U $ such that $\frac{\vx^T A \vx}{\vx^T \vx} \geq \lambda_j$.
Indeed, the space $V = \{ \vx \in \R^n | \langle \vx , \vv_i \rangle = 0 \text{ for $ i = j+1, \ldots, n$ } \}$ has codimension $n-j$, so the intersection $U\cap V$ is non-trivial.
Let $\vx \in U \cap V$, then $\alpha_i = \langle \vx , \vv_i \rangle = 0 $ for $i = j+1, \ldots, n$, thus 
$$\frac{\vx^T A \vx}{\vx^T \vx} = \frac{\sum_{i=1}^j \alpha_i^2 \lambda_i}{\sum_{i=1}^j \alpha_i^2} \geq \frac{\sum_{i=1}^j \alpha_i^2 \lambda_j}{\sum_{i=1}^j \alpha_i^2}= \lambda_j\, . $$

To show that  $\lambda_j \geq m$, we find a vector space $U$ such that $\dim U = n-j+1 $ and $M(U) \leq \lambda_j$.
Indeed, let $U = \spn \{\vv_i\}_{i=j}^n$.
It is imediate that $\dim U = n-j+1$ and since any $\vx \in U $ can be written as $\vx = \sum_{i=j}^n \alpha_i \vv_i$, we have that 
$$\frac{\vx^T A \vx}{\vx^T \vx} = \frac{\sum_{i=j}^n \alpha_i^2 \lambda_i}{\sum_{i=j}^n \alpha_i^2} \leq \frac{\sum_{i=j}^n \alpha_i^2 \lambda_j}{\sum_{i=j}^n \alpha_i^2}= \lambda_j\, . $$

This concludes the proof.
\end{proof}

\begin{lm}[Interlacing eigenvalues]\label{lm:interlacing}
Let $A$ be a $(n+r)\times (n+r)$ symmetric matrix and $B$ its upper left $n\times n$ submatrix, that is for any $i, j \in \{1, \ldots, n\}$ we have $A_{i, j} = B_{i, j}$.

Note that $B$ is also symmetric matrix, so let $\spec \, A = \{\lambda_1 \geq \cdots \geq \lambda_{n+r}\}$ and $\spec \, B = \{\mu_1 \geq \cdots \geq \mu_n\}$.

We have
$$ \lambda_j \geq \mu_j \geq \lambda_{j+r}\, . $$
\end{lm}

\begin{proof}
We use \cref{lm:maximalquad_EV}.
Specifically, to show that $\lambda_j \geq \mu_j$ we show that for any vector space $U\subseteq \R^{n+r}$ with $\dim U = j$ there exists $U'\subseteq \R^n$ with $m(U) \geq m(U')$.

\end{proof}

We now relate $\Delta(G)$, the largest degree of $G$, with the spectrum of matrices.

\begin{lm}\label{lm:adjacency_delta}
Let $G$ be a graph, $A_G$ be its adjacency matrix and $B$ a symmetric matrix obtained from $A_G$ by flipping the sign of some of its entries.

If $\lambda $ is the largest eigenvalues of $B$, then $\Delta(G) \geq \lambda $.
\end{lm}

\begin{proof}
Let $\vv $ be the eigenvector corresponding to $\lambda$, and let $i$ be the index such that $|\vv_i |$ is maximal.
We necessarily have $\lambda \geq \tr A = 0$.
Thus we have:
\begin{align*}
\lambda  |(\vv)_i| = |(\lambda \vv)_i| &= | (B\vv)_i | = \left| \sum_{j: i\-- j} B_{i, j} \vv_j  \right| \\
&\leq \sum_{j: i\-- j} |B_{i, j}| |\vv_j| = \sum_{j : i\-- j} |\vv_j| \leq  \sum_{j : i\-- j} |\vv_i| \leq \Delta(G) |\vv_i|\, . 
\end{align*}
Note that because $\vv$ is non trivial, we have $|\vv_i|  > 0$.
Therefore $\lambda \leq \Delta(G)$.
\end{proof}

\begin{proof}[Proof of \cref{thm:zhang}]

Note that the inductive structure on $C_n$ allows us to describe the adjacency matrix inductively.
Specifically,
$$ A_{C_1} = \begin{pmatrix}
0 & 1 \\
1 & 0
\end{pmatrix} \quad \quad \quad 
A_{C_{n+1}} = \begin{pmatrix}
A_{C_n} & \Id_{2^n} \\
\Id_{2^n} & A_{C_n}
\end{pmatrix}\, . $$

Define the matrices $\{B_n\}_{n\geq 1}$ so that $B_n$ results from $A_{C_n}$ by flipping some signs:
$$ B_1 = \begin{pmatrix}
0 & 1 \\
1 & 0
\end{pmatrix} \quad \quad \quad 
B_{n+1} = \begin{pmatrix}
B_n & \Id_{2^n} \\
\Id_{2^n} & -B_n
\end{pmatrix}\, ,$$
which gives us that $B_1^2 = \Id_{2}$ and 
$$
B_{n+1}^2 = \begin{pmatrix}
B_n^2 + \Id_{2^n} & 0 \\
0 & B_n^2 + \Id_{2^n}
\end{pmatrix}\, . $$
Thus, by an inductive argument we have that $B_n^2 = n\Id_{2^n}$.
Therefore, we conclude that $\spec\,~B_n~=~\{ \underbrace{\sqrt{n}, \ldots, \sqrt{n}}_{a \text{ times}}, \underbrace{-\sqrt{n}, \ldots, -\sqrt{n}}_{b \text{ times}} \}$. 

Furthermore, $\tr B_n = 0 $ so $a\sqrt{n} - b\sqrt{n} = 0$ and $a+b = 2^n$.
We conclude that we have $\spec \, B_n = \{ \underbrace{\sqrt{n}, \ldots, \sqrt{n}}_{2^{n-1} \text{ times}}, \underbrace{-\sqrt{n}, \ldots, -\sqrt{n}}_{2^{n-1} \text{ times}} \}$. 

Let $W \subseteq V(C_n)$ be a collection of vertices of $C_n$ with $|W| > 2^{n-1}$.
This corresponds to a $|W|\times |W|$ submatrix of $B_n$, call it $D$.
After rearranging the columns of the matrix, this can be identified with the upper right submatrix of $B_n$.
Furthermore, $D$ results from the adjacency matrix of $C_n|_W$ by flipping some signs, so \cref{lm:adjacency_delta} gives us that $\Delta(W) \geq \lambda$, where $\lambda $ is the highest eigenvalue of $D$.
But by \cref{lm:interlacing}, $\lambda \geq \lambda_{2^n - |W| + 1}(B_n) = \sqrt{n}$, because $2^n - |W| + 1 \leq 2^{n-1}$, as desired.
\end{proof}

%

\subsection{Consistent colouring}

\begin{figure}[h]
\includegraphics[scale=1]{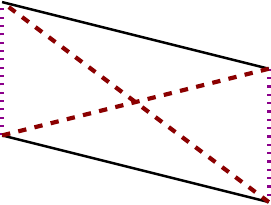}
\caption{A consistent colouring locally has either six colours or three colours distributed in this way.\label{fig:consistent_colouring}}
\end{figure}

An edge colouring of a graph $G = (V, E)$, or simply a colouring of $G$, is a map $f:E \to [k]$.
A colouring of a $K_n$ is called consistent if for any four vertices, the six edges among these vertices either have all distinct colours, or have three different colours as in \cref{fig:consistent_colouring}.

\begin{thm}
The graph $K_n$ has a consistent colouring with at most $n-1$ colours if and only if $n = 2^t $ for some integer $t$.
\end{thm}

\begin{proof}
Let $\CC \coloneqq\{ c_1, \dots, c_k\}$ be the collection of colours.
First, because each vertex has $n-1$ neighbours, and no two neightbour edges can have the same colour, we have $k \geq n-1$.
We are given that $k \leq n-1$, thus $k = n-1$. 
We will endow $\CC \cup \{\vec{0}\}$ with a structure of $\F_2$ vector space.

Incidentally, let $c$ be the colour of an edge $v_1 \-- v_2$, and $d$ be the colour of an edge $v_1 \-- v_3$.
We define $c + d $ to be the colour of the edge $v_1 \-- v_3$.
That this is well defined follows from a simple application of the consistency property in three squares.
This shows that $|\CC \cup \{\vec{0}\}| = n$ is a power of two.

For the converse result, we set $V = \F_2^n$ and colour an edge $\vv \-- \vw$ with the colour $\vv - \vw$.
It is a straightforward observation that this is a consistent colouring.
\end{proof}

\section{Geometry}

\subsection{Extremal geometry}

\subsubsection*{Silvester's Theorem}

\begin{prop}
Consider $X$, a set of $m$ points in the plane, not all colinear.
Then there are at least $m$ lines that go through at least two points in $X$.
\end{prop}

\begin{figure}[h]
\includegraphics[scale=.5]{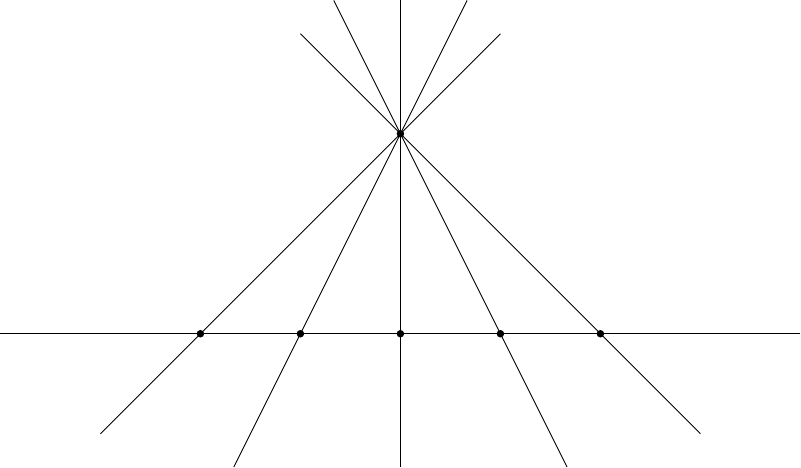}
\caption{The unique distribution of points that minimizes the number of lines.\label{fig:sylvester}}
\end{figure}

\begin{proof}
Let $\mathcal L$ be the collection of lines that go through at least two points in $X$.
For each line $\ell \in \mathcal L$, define the linear function in $\R^X$.
$$ f_{\ell}(\vx) = \sum_{x\in \ell} \vx_x \, .$$

Assume that $|\mathcal L | < m$.
Thus, there exists a non trivial solution to $f_{\ell}(\vx) = 0 $ for all $\ell \in \mathcal L$.
For such a solution $\vx$ we have
\begin{align*}
0 &= \sum_{\ell \in \mathcal L} \left( f_{\ell}(\vx) \right)^2 \\
&= \sum_{\ell \in \mathcal L} \left(\sum_{x \in \ell}\vx_x^2 \right) + \left( \sum_{x, y \in \ell} \vx_x \vx_y \right)\\
&= \sum_{x \in X} a_x\vx_x^2 + \sum_{x, y \in X}2 a_{x, y} \vx_x \vx_y  \, ,
\end{align*}
where $a_x$ is the number of lines $\ell \in \mathcal L$ that go through $x$, and $a_{x, y}$ is the number of lines that go through $x$ and $ y$.
It is clear that $a_{x, y} = 1$.
Furthermore, not all points are colinear, so $a_x\geq 2$ for each $x$.
We can rearrange the above expression as follows:
\begin{align*}
0 &=  \sum_{x \in X} a_x\vx_x^2 + \sum_{x, y \in X}2 a_{x, y} \vx_x \vx_y\\
&=  \sum_{x \in X} (a_x-  1) \vx_x^2 + \left(\sum_{x \in X}\vx_x \right)^2  \, .
\end{align*}

This is a sum of squares so it is only zero whenever each square is zero, that is $\vx = 0$, a contradiction.
We conclude that either all points are colinear, or $|\mathcal L | \geq m$.
\end{proof}

\subsubsection*{Kakeya problem}

A set $A \subseteq \F_q^n$ is called a \textbf{Kakeya set} if, for any non-zero $\vv\in\F_q^n$, there exists $\vw \in A$ such that $\forall_{\lambda \in \F_q}$ we have  $\vw + \lambda vv \in A$.

Interest in these sorts of sets arose with the original \textbf{Kakeya problem}, in \cite{fujiwara1917some}:

\begin{prob}[Kakeya problem]
Find a figure of least area on which a segment of length one can be turned through 360 degrees by a continuous movement.
\end{prob}

The solution to this problem, presented by \cite{besicovitch1928kakeya}, displayed, for each chosen $\varepsilon > 0$, a set with area smaller than $\varepsilon $ with the desired property, minus the continuous movement.
In the discrete world, we still require that segments in all directions are still in the Kakeya set, and it turns out that these will always be large.
This was established in \cite{dvir2009size}:

\begin{thm}[Dvir's solution to the discrete Kakeya problem]\label{thm:kakeya}
For all $n > 0 $ integer, there is a constant $c_n$ such that any Kakeya set $A \subseteq \F_q^n$ has $|A| \geq c_n q^n$.
\end{thm}

The best known $c_n$ is $2^{-n}$, we will show that $|A| \geq \binom{n+q-1}{n}$, which asymptotically gives $c_n = \frac{1}{n!}$.

\begin{lm}\label{lm:rootcount}
Let $f \in \F_q[x_1, \ldots, x_n]$ be a non-trivial polynomial of degree $d$.
Then there are at most $d q^{n-1}$ tuples $(z_1, \ldots, z_n)$ such that $f(z_1, \ldots, z_n) = 0$.
\end{lm}

\begin{proof}
We act by induction on $n$.
For the base case $n = 1$ we use induction again on $d$.
If $\deg f = 1$, then $f$ is a linear polynomial in only one variable, so it has a unique zero, so the theorem holds.

If $\deg f = d+1$, then either $f$ has no roots (in which case the theorem holds), or it has some root $\alpha$.
We use polynomial long division to wrote $f(x) = g(x) (x-\alpha) + r$, for some $g$ polynomial with $\deg(g) = d$ and $r \in \F_q$.
By substituting $x \mapsto \alpha$, we note that $r = 0$.
Furthermore, $g$ has at most $d$ roots by induction hypothesis, so $f(x) = (x-\alpha) g(x) $ has at most $d+1$ roots.

Now for the induction step on $n$, assume that $f$ is a polynomial in $n+1$ variables, and write
$$ f(x_1, \ldots, x_{n+1} ) = x_{n+1}^{d'} g_{d'}(x_1, \ldots, x_n) + x_{n+1}^{d'-1} g_{d'-1}(x_1, \ldots, x_n) + \cdots + g_0(x_1, \ldots, x_n)\, . $$

For a tuple $(z_1, \ldots, z_{n+1})$ such that $f(z_1, \ldots, z_{n+1}) = 0$ there are two possibilies:
\begin{enumerate}
\item The tuple $(z_1, \ldots, z_{n})$ vanishes on $g_{d'}$, that is $g_{d'}(z_1, \ldots, z_{n}) = 0$.

\item The tuple $(z_1, \ldots, z_{n})$ does not vanish on $g_{d'}$, and the scalar $z_{n+1} \in \F_q$ vanishes on the $d'$ degree polynomial $p_{z_1, \ldots, z_{n}}(x) \coloneqq \sum_{i=0}^{d'} a_i x^i$, where $a_i = g_{i}(z_1, \ldots, z_n)$.
\end{enumerate}

We count the tuples that satisfy each item.
Observe that $\deg g_{d'} \leq d - d'$.
By induction assumption on $n$, the number $M$ of tuples $(z_1, \ldots, z_{n})$ that vanish on $g_{d'}$ is at most $(d - d')q^{n-1}$, note that the number of tuples $(z_1, \ldots, z_{n})$ considered in item 1 is $ M q$.

For each tuple $(z_1, \ldots, z_{n})$ that does not vanish on $g_{d'}$, the polynomial $p_{z_1, \ldots, z_{n}}$ is a degree $d'$ polynomial so it has at most $d'$ zeroes.
It follows that there are at most $(q^n - M)d'$ tuples $(z_1, \ldots, z_{n+1})$ considered on item 2.

We conclude that there are at most 
\begin{align*}
M q + (q^n - M)d' = d' q^n + M(q - d')&\leq d'q^n + (d - d')q^{n-1}(q - d') \\
 &= d' q^n + (d - d')q^{n} - d'(d-d')q^{n-1} \\
 &= d q^n - d'(d-d')q^{n-1} \leq d q^n \, ,
\end{align*}
 zeroes - note that $d' \leq q$ - as desired.
\end{proof}

\begin{lm}\label{lm:monomialcount}
The number of monomials of degree at most $d$ in $n$ variables is $\binom{d+n}{n}$.
\end{lm}

\begin{proof}
This is the classical \textit{stars and bars} argument.
Given $d+n$ boxes alligned horizontally, there are $\binom{d+n}{n}$ ways of selecting $n$ many boxes.
We will construct a bijection between each selection and a monomial of degree at most $d$.

\begin{figure}
\includegraphics[scale=.5]{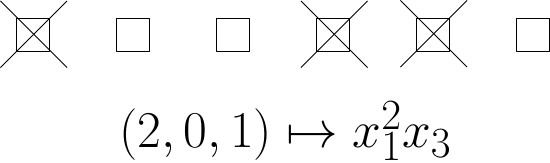}
\caption{A selection of boxes for $n=d=3$\label{fig:starsnbars}}
\end{figure}

Specifically, if $b_1 < b_2 < \cdots < b_n$ are the positions of the selected boxes, and $a_1, \ldots, a_{n+1}$ is the number of boxes in between the selected boxes, so that $a_i = b_i - b_{i-1} - 1$ (interpreting $b_0 = 0$ and $b_{n+1} = d+n+1$), then we have both $a_i \in \Z_{\geq 0}$ and $\sum_i a_i = d$.
The construction of the $a_i$ is displayed in \cref{fig:starsnbars} for $d = n = 3$.
To the tuple $(a_i)_{i=1}^n$ it corresponds the mononmial $x_1^{a_1}\cdots x_n^{a_n}$.
This mapping is invertible, and this concludes the bijective proof.
\end{proof}

\begin{proof}[Proof of \cref{thm:kakeya}]
The case $n=1$ is easy to deal with, so let us assume that $n \geq 2$.
Let $A$ be a Kakeya set, and assume for sake of contradiction that $|A| < \binom{q+n-1}{n}$.
The vector space $V$ of polynomials in $\F_q[x_1, \ldots, x_n]$ of degree at most $q-1$ has dimension $\binom{q+n-1}{n}$, according to \cref{lm:monomialcount}.

We now argue that there is some non-trivial polynomial $f\in V$ such that $f(a) = 0 $ for all $a \in A$.
Indeed, because $|A| < \binom{q+n-1}{n}$, and each condition $f(a) = 0 $ is a linear condition, we have that $\dim \{ f \in V | f(a) = 0 \forall_{a \in A} \} > 0$.

Let $k \coloneqq \deg f\leq q-1$, and let $f_k$ be the \textit{top component of} $f$, that is the terms in $f$ with degree $k$.
This polynomial has at most $k q^{n-1}$ roots, from \cref{lm:rootcount}, so because $n>1$, there is some non-zero vector $\vv \in \F_q^n$ such that $f(\vv) \neq 0$.
However, by the Kakeya property, there is some $\vw \in \F_q^n$ such that $\vw + \lambda \vv \in A $ for any scalar $\lambda \in \F_q$.
Thus, $f(\vw + \lambda \vv ) = 0$ for any $\lambda \in \F_q$.

However, as a polynomial expression in $\lambda $, the polynomial $f(\vw + \lambda \vv )= f(\vv)_k \lambda^k + \mathcal O (\lambda^{k-1})$ has degree $k \leq q-1$, so according to \cref{lm:rootcount} it has at most $q-1$ roots.
Therefore, there is some $\lambda $ such that $(\vw + \lambda \vv ) \neq 0$, a contradiction.
\end{proof}

\subsubsection*{Two-distance sets}

Let $D$ be a finite set of positive real numbers.
A set $S\subseteq \R^d$ is called a $D$-distance set if for any two distinct points $\vx, \vy\in S$ we have $||\vx - \vy || \in D$.

If $D = \{ \delta \}$, then $S$ can only have $d+1$ points, and this bound is tight: for instance take the vertices of the $d$ simplex.
If $D$ has size two, we do not know what is the maximal size of $S$.

\begin{smpl}
Let $V = \{ \vv \in \{0, 1\}^{n+1} | \sum_{i=1}^{n+1} \vv_i = 2 \}$.
This is a set with $|V| = \binom{n+1}{2}$.
It is clear that $V\in \R^{n+1}$.
Note however that $V$ lies on the $n$-dimensional hyperplane $\sum_i \vv_i = 2$, so there is an isometric map sending $V$ to $\R^n$.

We can see as well that this is a $\{\sqrt{2}, 2\}$-distance set: if $\vx, \vy \in V$ are distinct vectors, then they either differ in two or four entries, depending on whether the entries with $1$ overlap or not.
In one case we have $||\vx - \vy || = \sqrt{2}$ and in the other case we have $||\vx - \vy || = \sqrt{4} = 2$.

This tells us that there is a $\{\sqrt{2}, 2\}$-distance set in $\R^n$ with $\binom{n+1}{2}$ many elements.
\end{smpl}

\begin{thm}
If $S\subseteq \R^n$ is a $\{\delta_1, \delta_2 \}$-distance set, then $|S| \leq \frac{(n+1)(n+4)}{2}$.
\end{thm}

\begin{proof}
For each $\vv \in S$, construct the following polynomial $f_{\vv} \in \R[x_1, \ldots, x_n]$:
$$f_{\vv} (\vx) \coloneqq \left(||\vv - \vx||^2 - \delta_1^2\right)\left(||\vv - \vx||^2 - \delta_1^2\right) \, . $$

Note that for $\vv \neq \vw $ elements of $S$, we have $f_{\vv} (\vw) = 0$ and $f_{\vv}(\vv) = \delta_1^2\delta_2^2$.

It can be seen that $f_{\vv}$ is in the vector space
$$\spn \{\left(\sum_{i=1}^n\right)^2, x_k \sum_{i=1}^n x_i^2, x_k^2, x_kx_j, x_k, 1 | k, j \in \{1, \ldots, n\} \}\, ,$$
which has dimension $1 + n + n + \binom{n}{2} + n + 1 = \frac{(n+1)(n+4)}{2}$.

On the other hand, $\{f_{\vv}\}_{\vv \in S}$ is a linearly independent set of polynomials: if we take a linear combination such that $0 = \sum_{\vv \in S} \alpha_{\vv} f_{\vv}$, by evaluating at some $\vw \in S$ we get
$$ 0 = \sum_{\vv \in S} \alpha_{\vv} f_{\vv}(\vw) = \alpha_{\vw} f_{\vw}(\vw) = \alpha_{
\vw} \delta_1^2 \delta_2^2 \Rightarrow \alpha_{\vw} = 0 \, ,$$
which means that $\{f_{\vv}\}_{\vv \in S}$ is a linearly independent set.
\end{proof}

\subsubsection*{Joints Problem}

A \textbf{joint} in a collection $\mathcal L$ of lines in $\R^3$ is an intersection of at least three non-coplanar lines.

\begin{thm}\label{thm:joints}
Tere exists a constant $C$ such that given $N$ lines in $\R^2$ forming $J$ joints we have that
$$ J \leq C N^{3/_2}\, ,$$
furthermore, this bound is tight.
\end{thm}

\begin{smpl}
This example shows that the bound in \cref{thm:joints} is tight when $C = 3^{-3/_2}$.
Specifically, take an integer $n$, we can find $3n^2$ lines that intersect in $n^3$ joints. 
For $n=3$ we can see the example in \cref{fig:joints}.
The general construction is as follows: we take the collection of $n^2$ lines with direction $(0, 0, 1)$ that go through $(a, b, 0)$, where $a, b\in [n]$, $n^2$ lines with direction $(0, 1, 0)$ that go through $(a, 0, b)$, where $a, b\in [n]$, $n^2$ lines with direction $(1, 0, 0)$ that go through $(0, a, b)$, where $a, b\in [n]$.

In this line configuration, any point $(a, b, c) $ with $a, b, c \in [n]$ is a joint, so $J = n^3$, whereas $N = 3 n^2$.

\begin{figure}[h]
\includegraphics[scale=.6]{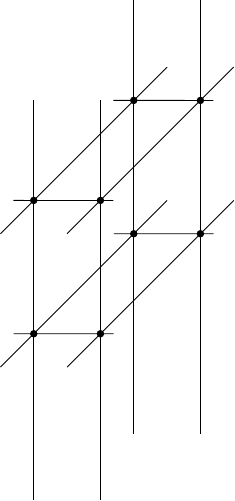}
\caption{A construction of a collection of lines with high number of joints.\label{fig:joints}}
\end{figure}
\end{smpl}

\begin{proof}
Now, we establish the inequality for $C = 3^{3/_2}$.
First, we show that for any collection of lines $\mathcal L$, there is a line with at most $3 J^{1/_3}$ joints.

Once this is established, the theorem follows by iteratively removing a line from $\mathcal L$.
Each removal takes away at most $3 J^{1/_3}$ joints with it.
Once this is done to completion, so that there are no more lines left, we have removed at most $N 3 J^{1/_3}$ joints, so $J \leq 3 N J^{1/_3} \Rightarrow J \leq (3N)^{3/_2}$.

Acting by contradiction, assume that there is no line with at most $3 J^{1/_3}$ joints. 
We consider a polynomial $p \in \R[x, y, z]$ such that 
\begin{enumerate}
\item It is non-zero;

\item It vanishes at all joints;

\item It has degree \textbf{on each variable} at most $J^{1/_3}$.
\end{enumerate}

A polynomial satisfying item 3 can be written as a combination of $(\lfloor J^{1/_3}\rfloor + 1)^3 > J$ monomials, item 2 amounts to $J$ linear equations, so by dimension analysis such a non-zero polynomial exists.
We pick $p$ that minimizes the degree.

The polynomial $p$ restricted to any line is a polynomial of degree at most $3J^{1/_3}$.
Because each line has, by contradiction hypothesis, at least $3 J^{1/_3} + 1$ joints, this polynomial must be identically zero in each line.

Consider the polynomials
$$q_x \coloneqq \frac{\partial}{\partial x} p \, , q_y \coloneqq \frac{\partial}{\partial y} p \, , q_z \coloneqq \frac{\partial}{\partial z} p \, .$$
It follows from the previous paragraph that all directional derivatives of $p$ at each joint $\vj$ vanish, so $p_x(\vj) = p_y(\vj) = p_z(\vj) = 0$.

Furthermore, $p_x, p_y, p_z$ also satisfy item 3, as these polynomials have smaller degree than $p$.
By minimality of $p$, we must have $p_x \equiv p_y \equiv p_z  \equiv 0$, which means that $p$ is the constant polynomial.
This is a contradiction.
\end{proof}
%
%
%
%
%
%

\subsubsection*{Nearly orthogonal vectors}

A set $Y \subseteq \R^d$ is said to be \textbf{orthogonal} if any two distinct points $\vx, \vy \in Y$ have $\vx \perp \vy$.
A set $X\subseteq \R^d$ is said to be \textbf{nearly orthogonal} if any three distinct points $\vx, \vy, \vz \in X $ have either $\vx\perp \vy, \vy\perp\vz$ or $\vz\perp\vx$.

\begin{thm}[Rosenfeld, 1991]\label{thm:nearly}
A \textbf{nearly orthogonal set} $X\subseteq \R^d$ has $|X| \leq 2d$.
\end{thm}

We first establish a few important Lemmas.

\begin{lm}[Paserval]\label{lm:paserval}
For a set $Y \subseteq \R^d$ of \textbf{orthogonal vectors} with unit length, and $\vv \in \R^d$, we have 
$$ \sum_{\vy \in Y} \langle \vv, \vy\rangle^2 \leq ||\vv ||^2 \, . $$
\end{lm}

\begin{proof}
We start by extending $Y = \{\vy_1, \ldots, \vy_k\}$ to a base $\{\vy_1, \ldots, \vy_d\}$ of $\R^d$.
This is possible because non-zero orthogonal vectors form a linearly independent set. 
This extension can further be orthogonalized via the Graham-Schmidt algorithm.

In this way, for a vector $\vv\in \R^d$ we have $\vv = \sum_{i=1}^d \alpha_i \vy_i$, where for each $i=1, \ldots, d$ it can be observed that $\alpha_i = \langle \vv, \vy_i\rangle$.
Thus 
\begin{align*}
|| \vv ||^2 &= \langle \vv, \vv \rangle  = \sum_{i=1}^d \sum_{j=1}^d \alpha_i \alpha_j  \langle \vy_i,  \vy_j \rangle \\ 
&= \sum_{i=1}^d \alpha_i^2 \geq \sum_{i=1}^k \alpha_i^2 = \sum_{i=1}^k \langle \vv, \vy_i\rangle^2,
\end{align*}
as desired.
\end{proof}

\begin{lm}\label{lm:rank_ineq}
If $A$ is a symmetric matrix, we have 
$$ \rk \, A \geq \frac{(\tr A)^2}{\tr (A^2)} \,  . $$
\end{lm}

\begin{proof}
Let $r = \rk A$, which is the number of non-zero eigenvalues of a matrix $A$, with multiplicity.
Write $\lambda_1, \ldots, \lambda_r$ for the non-zero eigenvalues of $A$.

Recall from Cauchy inequality we have
$$\left( \underbrace{1 + \cdots + 1}_{r \text{ times}} \right)\left(\sum_{i=1}^r \lambda_i^2 \right) \geq \left( \sum_{i=1}^r \lambda_i \right)^2\, . $$

With this we have $ r \, \tr (A^2) = r \sum_{i=1}^r \lambda_i^2 =  \left( \sum_{i=1}^r \lambda_i \right)^2 = (\tr A)^2 $ which can be rearranged to the desired inequality.
\end{proof}

\begin{proof}[Proof of \cref{thm:nearly}]
Let $X = \{\vx_1, \ldots, \vx_m\}\subseteq \R^d$ be a nearly orthogonal set and assume all vectors are unit vectors.
We wish to show that $m \leq 2 d$.
Let $B = \begin{pmatrix}
| &  & | \\ \vx_1 & \cdots & \vx_d \\ | & & | \end{pmatrix}$ and $A = B^T B$.
Note that $\rk A \leq \rk B\leq d$ and $\tr A = \sum_{i=1}^m \langle \vx_i, \vx_i\rangle = m$.

However,
\begin{align*}
\tr (A^2) &= \sum_{i=1}^m (A^2)_{i, i} = \sum_{i=1}^m \sum_{j=1}^m A_{i, j} A_{j, i} = \sum_{i=1}^m \sum_{j=1}^m \langle \vx_i, \vx_j \rangle^2 \\
&= \sum_{i=1}^m \left(\langle \vx_i, \vx_i \rangle^2 \right)+ \sum_{i=1}^m\left( \sum_{j : \vx_j \perp \vx_i } \langle \vx_i, \vx_j \rangle^2 \right) + \sum_{i=1}^m\left( \sum_{\substack{j : \vx_j \not\perp \vx_i \\ j \neq i}} \langle \vx_i, \vx_j \rangle^2 \right) \\
&= m + \left( \sum_{j : \vx_j \perp \vx_i } \langle \vx_i, \vx_j \rangle^2 \right) + \left( \sum_{\substack{j : \vx_j \not\perp \vx_i \\ j \neq i}} \langle \vx_i, \vx_j \rangle^2 \right)\\
&= m + \left( \sum_{\substack{j : \vx_j \not\perp \vx_i \\ j \neq i}} \langle \vx_i, \vx_j \rangle^2 \right) \leq m + \sum_{i=1}^m ||\vx_i ||^4  = 2m  \, , 
\end{align*}
where we use \cref{lm:paserval}, for each $i$, on the orthonormal set $\{\vx_j : \vx_j \not\perp \vx_i\}$.

Together with \cref{lm:rank_ineq}, this gives us $ 2m \geq \tr(A^2) \geq \frac{(\tr A )^2}{\rk A}\geq \frac{m^2}{d}$, which can be rearranged to $m \leq 2d $.
\end{proof}

\subsubsection*{Efficient projection}

\begin{thm}[Johnson - Linderstrauss Theorem]\label{thm:JL}
Fix $d >> 1$ and $\varepsilon \in (0, 1)$.
Let $P \subseteq \R^d$ with $|P| = n$ finite.

Then there exists a linear map $A:\R^d \to \R^k$, with $k = C_{\varepsilon} \log (n)$, such that for any two distinct $\vx, \vy \in P$ we have
$$ 1 - \varepsilon \leq \frac{||A(\vx) - A(\vy)||}{||\vx - \vy||} \leq 1 + \varepsilon \, . $$
\end{thm}

\begin{proof}[Proof of \cref{thm:JL}]
Probabilistic method with a random matrix
\end{proof}

\begin{lm}\label{lm:JL_1}
Fix $n$ integer and $\varepsilon \in (0, 1)$.
Let $A = [a_{i, j}]_{\substack{i = 1, \ldots , n \\ j = 1, \ldots , n}}$ be a symmetric matrix.
Assume that $a_{i, i} = 1$ and $|a_{i, j}| < \varepsilon $ for all $i, j$ distinct in $[n]$.

Then $\rk \,A \geq n \frac{1}{1 + (n-1)\varepsilon^2}$.
In particular, if $\varepsilon < (n-1)^{-0.5}$, then $\rk A \geq 0.5 n$ and if $\varepsilon < (n-1)^{-1.5}$ then $\rk \,A = n$.
\end{lm}

Recall that $\rk \, A$ is the number of eigenvalues of $A$ that are non zero.
Recall as well that $\tr \, A$ is the sum of the eigenvalues.
It is a known fact that this is the same as the sum of the diagonal entries of the matrix $A$.

\begin{proof}
Let $\lambda_1, \ldots, \lambda_n$ be the eigenvalues of $A$, with multiplicity.
Because $A$ is symmetric, all eigenvalues are real numbers, so we have from Cauchy-Schwarz inequality that
$$\rk \, A \,  \tr (A^2)= \left(\sum_{\lambda_i \neq 0} 1 \right)\left( \sum_{\lambda_i \neq 0} \lambda_i^2\right) \geq \left( \sum_{\lambda_i \neq 0} \lambda_i \right)^2 = (\tr\,  A)^2 = n^2 \, . $$

Rewritting and using that $|a_{i, j}| < \varepsilon$ for $i\neq j$ we have 
\begin{align*}
\rk \, A &\geq n^2/\tr (A^2) = \frac{n^2}{\sum_{i, j} a_{i, j}^2} \\
&\geq \frac{n^2}{n + n(n-1)\varepsilon^2} = \frac{n}{1 + (n-1)\varepsilon^2}\, ,
\end{align*}
which is the desired inequality.
\end{proof}

\begin{lm}\label{lm:JL_2}
Fix a positive integer $k$ and let $A = [a_{i, j}]_{\substack{i = 1, \ldots , n \\ j = 1, \ldots , n}}$ be a symmetric matrix.
Let $B = [a_{i, j}^k]_{\substack{i = 1, \ldots , n \\ j = 1, \ldots , n}}$.
Then $\rk B \leq \binom{k + \rk A - 1}{k}$.
\end{lm}

\begin{proof}
Let $\vv^{(1)}, \ldots, \vv^{(r)} $ be a basis of $\mathrm{rowspace}(A)$, where $r = \rk A$.
Then $\vw \in \mathrm{rowspace}(B)$ is in 
$$V \coloneqq \spn \left\{ \left((\vv_j^{(1)})^{t_1} \cdots (\vv_j^{(r)})^{t_r} \right)_{j = 1, \ldots, n}\right\}_{(t_1, \ldots , t_t) \in \Z_{\geq 0}^r | \sum_i t_i = k } \, .$$

It is easy to see that $\dim V = \binom{k + r - 1}{k}$, which leads to the desired result.
\end{proof}

\begin{lm}\label{lm:JL_3}
Fix $n$ integer and $\varepsilon \in (0, 0.5)$.
Let $B = [b_{i, j}]_{\substack{i = 1, \ldots , n \\ j = 1, \ldots , n}}$ be a matrix, not necessarily symmetric.
Assume that $b_{i, i} = 1$ and $|b_{i, j}| < \varepsilon $ for all $i, j$ distinct in $[n]$.

If $\varepsilon > (n-1)^{-.5}$, then there exists a constant $C$ independent of $n$ and $\varepsilon$ such that 
$$ \rk B \geq \frac{C \log n}{\varepsilon^2 \log (\varepsilon^{-1})}\, . $$
\end{lm}

Note that the case where $\varepsilon < (n-1)^{-.5}$ has already been covered in \cref{lm:JL_3}.

\begin{proof}
Note that 
$$\rk(B + B^T) \leq \dim( \mathrm{rowspace}(B) ) + \dim( \mathrm{rowspace}(B^T) ) =  \rk B + \rk B^T  = 2 \, \rk \, B\, . $$ 
Therefore, it is sufficient to show the result for a symmetric matrix $B$.

We split the remaning part of the proof in two cases: when $2 \log(\varepsilon^{-1}) \leq \log n$ and when $2 \log(\varepsilon^{-1}) \geq \log n$.

In the first case define $k = \lfloor \frac{\log n}{2 \log ( \varepsilon^{-1})} \rfloor$ and $n' = \lfloor \varepsilon^{-2k} \rfloor$, in such a way that 
$$ n' \leq \frac{1}{\varepsilon^{2k}} \leq \left( \frac{1}{\varepsilon} \right)^{\frac{\log n}{2 \log(\varepsilon^{-1})}} = n\, . $$
Furthermore, $\varepsilon^k \leq \frac{1}{\sqrt{n}} \leq \frac{1}{\sqrt{n-1}}$.
Let $B_0$ be the $n'\times n'$ submatrix of $B$ arising from the first $n'$ rows and $n'$ columns of $B$, and let $B^{(k)}$ be the matrix $B_0$ with each entry raised to the $k$-th power.
Note how $B^{(k)}$ is still a symmetric function.
From \cref{lm:JL_1} we have that $\rk \, B^{(k)} \geq .5 \, n'$.

On the other hand, \cref{lm:JL_2} gives us that 
$$\rk \, B^{(k)} \leq \binom{k + \rk\, B_0 - 1}{k} \leq e^k \left( \frac{k + \rk \, B_0}{k}\right)^k =  e^k \left(1 +  \frac{\rk \, B_0}{k}\right)^k\, , $$
where we are using the bound presented in \cref{lm:binom_bound}.

Puting it all together and using $n' \leq \varepsilon^{-2k}$ we get 
$$ \rk \, B_0 \leq k \left( \frac{1}{2^{1/_k} e} \sqrt[k]{n'} - 1\right)  \leq  \frac{k}{2^{1/_k} e} \sqrt[k]{n'} \leq \frac{1}{2e  2^{1/_k} }\frac{\log n}{\log(\varepsilon^{-1})}\varepsilon^{-2} \leq \frac{C \log n}{\varepsilon^2 \log(\varepsilon^{-1})}\, , $$
where $C = \frac{1}{2e}$.

For the other side, note that $2 \log(\varepsilon^{-1}) \geq \log n$ implies $\varepsilon \leq n^{-.5}$
\end{proof}

\subsection{Convex geometry}

A set $C \subseteq \R^d$ is said to be \textbf{convex} if for any points $\vx, \vy \in C$ and a real $\lambda \in [0, 1]$ we have $t\vx + (1-t)\vy \in C$.
The \textbf{convex hull} of a set $S$ is the smallest convex set $\conv S $ that contains $S$.
Because the convexity condition is closed for arbitrary intersections, it can be defined as $\conv S \coloneqq \bigcap_{S \subseteq C \text{ convex}} C$ or equivalently
$$ \conv S = \{ \sum_{i=1}^m \alpha_i\vv_i | \vv_i \in S, \, \alpha_i \in [0, 1], \sum_i \alpha_i = 1\} \, . $$

\begin{thm}[Caratheodory Theorem]\label{thm:caratheodory}
Let $S \subseteq \R^d$.
If $\vv \in \conv (S)$ then $\vv = \sum_{i=1}^{d+1} \alpha_i \va_i$ for some $\va_i \in S, \alpha_i \in \R_{\geq 0}$, such that $\sum_i \alpha_i = 1$.
\end{thm}

\begin{proof}
Because $\vv \in \conv S$, there is an integer $m$, real numbers $\alpha_1, \ldots, \alpha_m $ and vectors $\vx_1, \ldots, \vx_m \in S $ such that $\vv = \sum_{i=1}^{m} \alpha_i \va_i$ and $\sum_i \alpha_i = 1$.
The ones that give us $m$ minimal, and assume for sake of contradiction that $m > d+1$.

For each $i \in [m]$, let $\vv_{i} = (\vx_i , 1)$.
Because $m > d+1$, the vectors $\{ \vv_i \}_{i=1}^m$ are linearly dependent, so there is some non-trivial vanishing linear combination $\sum_{i=1}^m \beta_i \vv_i = 0$.
Note that this gives us $\sum_{i=1}^m \vx_i = 0 $ and $\sum_i \beta_i = 0$, therefore we have, for any real $t$, that $\vv = \sum_{i=1}^{m} (\alpha_i - t \beta_i) \va_i$ and $\sum_i ( \alpha_i  - t \beta_i ) = 1$.

The proof is concluded when we will find $t$ such that $\alpha_i - t \beta_i \geq 0$.
Indeed, let $t_0 = \min_{i : \beta_i > 0} \frac{\alpha_i}{\beta_i}$.
Because $\sum_i \beta_i = 0$ and the linear combination $\sum_{i=1}^m \beta_i \vv_i$ was taken to be non-trivial, there is some $\beta_i > 0$, so $t_0 $ is well defined and positive.
The minimality condition guarantees that for each $i$ we have $\alpha_i - t_0 \beta_i \geq 0$.
Furthermore, there is an index $I$ such that $\alpha_i - t_0 \beta_i = 0$.
This means we can write $\vv = \sum_{i=1, i \neq I}^{m} (\alpha_i - t \beta_i) \va_i$.
This contradicts the minimality of $m$.
\end{proof}

\begin{lm}[Radon's Lemma]\label{lm:radon}
If $S\subseteq \R^d$ with $|S| \geq d+2$ there exists a partition $S = L \uplus R$ such that $$\conv (L) \cap \conv (R) \neq \emptyset \, .$$
\end{lm}

We observe that this is tight.
Indeed, if we take $S$ to be $d+1$ points in general position, then no such partition can be found.

\begin{proof}
Let $S = \{\vv_1, \ldots, \vv_m \}$ with $m\geq d+2$.
Let $\vx_i = (\vv_i, 1)\in \R^{d+1}$.
These are linearly dependent vectors, so there is some non-trivial linear combination $\sum_{i=1}^m \beta_i \vx_i = 0$.
This means that $\sum_{i=1}^m \beta_i = 0$ and $\sum_{i=1}^m \beta_i \vv_i = 0$.

Let $I = \{i\in [m] | \beta_i  > 0 \}$ and $J = \{j\in [m] | \beta_j < 0 \}$.
Define $b = \sum_{i\in I} \beta_i = \sum_{j\in J} - \beta_j$.
Because $(\beta_i)_{i=1}^m$ is non-trivial, $B>0$.

Thus, we can take $\vec{s} = \sum_{i\in I} \frac{\beta_i}{B} \vv_i = \sum_{j\in J}\frac{-\beta_j}{B}\vx_i$.
This satisfies both $\vec{s} \in \conv \{\vx_i\}_{i\in I}$ and $\vec{s} \in \conv \{\vx_j\}_{j\in J}$, so these have non-empty intersections, as desired.
\end{proof}

\begin{thm}[Helly's theorem]\label{thm:helly}
Let $m, d $ be integers such that $m \geq d+1$.
Take $C_1, \ldots , C_m \subseteq \R^d$ convex sets such that every $d+1$ many such sets have non-empty intersection.
Then 
$$\bigcap_i C_i \neq \emptyset \, . $$
\end{thm}

\begin{proof}
We use induction on $m$.
The theorem is imediate for $m = d+1$, as the theorem condition is the same as the desired result.

We now prove the induction step.
Take $m$ and a collection $\{C_i\}_{i=1}^m$ of convex sets such that every $d+1$ many such sets have non-empty intersection, but $\bigcap_i C_i = \emptyset$.

By induction hypothesis, for each $j = 1, \ldots, m$ we have some $a_j \in \bigcap_{\substack{i=1 \\ i\neq j}}^m C_i$.
By \cref{lm:radon}, there is a partition $\{a_j\}_{j=1}^m = I \uplus J$ such that $\conv \, I \cap \conv \, J \neq \emptyset$.
Let $\vx \in \conv \, I \cap \conv \, J $.
We will show that $\vx \in \bigcap_i C_i$.

Assume with no loss of generality that $I = \{a_1, \ldots, a_l\} $ and $J= \{ a_{l+1} , \ldots, a_m\}$.
We have that $\vx \in C_i$ for $i = 1, \ldots, l$, as for these $i$ we have $J \subseteq  C_i$, so $\conv \, J \subseteq C_i$.
Symilarly, we have that $\vx \in C_i$ for $i = l+1, \ldots, m$, as for these $i$ we have $I \subseteq  C_i$, so $\conv \, I \subseteq C_i$.
We conclude that $\vx \in \bigcap_i C_i$
\end{proof}

Any hyperplane $h \subseteq \R^d$ disconnects $\R^d \setminus h$ into two components.
We write $h^+$ and $h^-$ for these open sets, and $\overline{h}^+$, $\overline{h}^-$ for their respective closures.
These closed sets are called \textbf{hyperspaces}.

\begin{defin}[Centerpoint]
Given a finite set $X \subseteq \R^d$, an $\alpha$-centerpoint $\vy \in \R^d$ is a point such that, for any hyperplane $h$ that contains $\vy$ we get
$$ |X \cap h^+|, |X\cap h^-| \geq \alpha |X| \, . $$
\end{defin}

\begin{thm}[Centerpoint theorem]\label{thm:centerpoint}
Every finite set $X$ has an $\alpha$-centerpoint for $\alpha = \frac{1}{d+1}$.
\end{thm}

\begin{obs}
This is tight.
Indeed, assume that $\alpha > \frac{1}{d}$, if we take $X$ to be a collection of $d+1$ points in $\R^d$ in general position, these generate $d+1$ distinct hyperplanes $D_1, \ldots, D_{d+1}$.
Any $\alpha$-centerpoint $\vy$ is contained in a hyperplane parallel to each $D_i$, and the condition $|X \cap h^+|, |X\cap h^-| \geq \alpha |X|$ forces this hyperplane to be precisely $D_i$.
However, there is no point $\vy \in \bigcap_i D_i = \emptyset$.
\end{obs}

\begin{proof}[Proof of \cref{thm:centerpoint}]
Consider $\mathcal X $ the collection of hyperspaces $\overline{h}^+$ that contain more than $\frac{d}{d+1}|X|$ points of $X$, and let 
$$ \mathcal Y \coloneqq \{ \conv \, (\overline{h}^+ \cap X) | \overline{h}^+ \in \mathcal X \}\, . $$

Note that $\mathcal Y $ is finite, as there are only finitely many subsets of $X$.
We claim that $\bigcap_{Y \in \mathcal Y} Y \neq \emptyset$.
Indeed, from \cref{thm:helly}, it is sufficient to show that for any $d+1$ sets $C_1, \ldots, C_{d+1} \in \mathcal Y$ we have $\bigcup_{i=1}^{d+1} C_i \neq \emptyset$.

\begin{align*}
|X \cap C_1\cap \cdots \cap C_{d+1} |  &= |X| - |X \cap (C_1 \cap \cdots \cap C_{d+1})^c | \\
&= |X| - |X \cap (C_1^c \cup \cdots \cup C_{d+1}^c) | \\
&= |X| - |(X \cap C_1^c) \cup \cdots \cup (X \cap C_{d+1}^c)|\\
&\geq |X| - \sum_i |X \cap C_i^c|\\
&> |X| - \sum_i \frac{1}{d+1} |X| = 0 \, ,
\end{align*}
where we note that $|X \cap C_i^c| < \frac{1}{d+1}|X|$ because, by assumption, $X \cap C_i = \overline{h}^+ \cap X$ contains more than $\frac{d}{d+1}|X|$ points.

Therefore we conclude that $\bigcup_{i=1}^{d+1} C_i \neq \emptyset$.
We claim that any point $\vy \in \bigcup_{i=1}^{d+1} C_i $ is a $\frac{1}{d}$-centerpoint.
Indeed, for sake of contradiction assume that there is some hyperplane $h$ through $\vy$ such that $|\overline{h}^+ \cap X|< \frac{1}{d+1}|X|$, then  $|h^+ \cap X|< \frac{1}{d+1}|X|$.
Because $X$ is a discrete set, we can find a parallel hyperplane $g$ such that $\overline{g}^+\cap X = h^+ \cap X$.
Note how  $\vy \not \in \conv \, (\overline{g}^+ \cap X)$, so $\vy \not \in \bigcup_{i=1}^{d+1} C_i $, a contradiction.
\end{proof}

\begin{thm}[Colorful Caratheodory Theorem]\label{thm:colorful_c}
Let $M_1, \ldots, M_{d+1} \subseteq \R^d$ finite sets such that $\va \in \bigcap_{i = 1}^{d+1} \conv (M_i) \neq \emptyset$.
Then, there exists $\vx_1, \dots, \vx_{d+1}$ such that $\vx_i \in M_i$ for all $i = 1, \ldots, d+1$ and $\va \in \conv (\{\vx_1, \ldots, \vx_{d+1}\})$.
\end{thm}

\begin{lm}\label{lm:closeness}
Let $H$ be a hyperplane, $\vx \in H$ and $\va, \vy$ on the same side of $H$ such that $\vx \-- \va$ is perpendicular to $H$.
Then there exists some $\lambda \in [0, 1]$ such that $(1-\lambda) \vx +\lambda \vy$ is closer to $\va$ than $\vx$.
\end{lm}

\begin{proof}
Let $f(\lambda)  = || (1-\lambda) \vx +\lambda \vy - \va ||^2 - ||\vx - \va||^2$.
Note that $f(0) = 0$.
We claim that $f'(0) < 0$, which means that for $\lambda >0 $ small enough, we have $f(\lambda) < 0$ and thus  $\va$ is closer to $(1-\lambda) \vx +\lambda \vy$ than $\vx$.
Indeed, note that

\begin{align*}
f(\lambda) &=  || (1-\lambda) \vx +\lambda \vy - \va ||^2 - ||\vx - \va||^2\\
&=  || (1-\lambda) \vx +\lambda \vy - \vx + \vx - \va ||^2 - ||\vx - \va||^2\\
&=  || \lambda(\vy - \vx ) + \vx - \va ||^2 - ||\vx - \va||^2\\
&=  || \lambda(\vy - \vx )||^2 + 2\langle \lambda(\vy - \vx ),  \vx - \va\rangle + || \vx - \va ||^2 - || \vx - \va ||^2 \\
&=   \lambda^2 ||(\vy - \vx )||^2 - 2\lambda \langle \vx - \vy,  \vx - \va\rangle\\
f'(0) &=  -2 \langle \vx - \vy,  \vx - \va\rangle < 0\, ,\\
\end{align*}
because the angle between the vectors $\vx - \vy,  \vx - \va$ is less than $\frac{\pi}{2}$.
\end{proof}

\begin{proof}[Proof of \cref{thm:colorful_c}]
Assume otherwise, and let $\vx_i \in M_i$ be chosen such that the distance between $\va $ and $X \coloneqq \conv(\{\vx_1, \ldots, \vx_{d+1}\})$ is minimal.
Because the sets $M_i$ are all finite, this minimum exists and is attained by some $\{\vx_i\}$.
For sake of contradiction assume that this distance is positive, so that there is some $\vx_0 \in X $ that is the closest point in $X$ to $\va$.

Let $H$ be the hyperplane perpendicular to $\vx_0 \-- \va$ that passes through $\vx$.
Let $H^+$ be the component of $\R^d\setminus H $ that contains $\va$.
If there exists some $\vy \in H^+ \cap X$, then we can construct a point $\vy \in X$ closer to $\va$ than $\vx$, which is impossible by minimality of $\vx_0$.
Therefore $H^+ \cap X = \emptyset$.

Now write $\vx_0 \in X \cap H$, which is a $d-1$-dimensional convex set.
By \cref{thm:caratheodory} we can write $\vx_0$ as a convex combination of $d$ terms, say $\sum_{i=1}^{d+1} \alpha_i \vx_i$ where $\alpha_j = 0$.
Recall that $\va \in \conv M_j$, so there is some $\hat{\vx}_j \in H^+$.
Replace $\vx_j$ by $\hat{\vx}_j$, and let $X' = \conv\{\vx_i\}$ be the new convex hull.
Note that from \cref{lm:closeness} there exists some point $\vv \in \conv\{\vx_0 \hat{\vx_j}\} \subseteq X'$ such that $\vv$ is closer to $\va$ than $\vx_0$.
This contradicts the minimality assumption and concludes the proof.
\end{proof}

The following is a generalisation of \cref{lm:radon}.

\begin{thm}[Tveberg's theorem]
Let $r$ be an integer, and $S \subseteq \R^d$ be such that $|S| \geq (r-1)(d+1) + 1$.
Then there exists a partition $S = S_1\uplus \cdots \uplus S_r$ such that $\bigcap_{i=1}^r \conv S_i \neq \emptyset$.
\end{thm}

\begin{proof}
Write $S = \{\vx_0, \ldots, \vx_m\}$.
It is enough to establish the result for $m = (d+1)(r-1)$.
Define $\vx_i \in \R^{d}$ for $i = 0, \ldots, m$ so that $(\vx_i)_j = \begin{cases}&1, \text{ if $j \in F_i$ } \\  &0, \text{ otherwise.}\end{cases}$, and let $\vy = (\vx, 1) \in \R^{d+1}$.
Define $\vv_i = \vec{e}_i \in\R^{r-1}$ for $i = 1, \ldots, r-1$, and define $\vv_r = - \sum_{i=1}^{r-1} \vec{e}_i$.
Write $\vec{0} \in\R^{d+1}$ for the all zero vector and $\mathbb{0} \in \R^{(d+1)\times (r-1)}$ for the all zero matrix.

Define $M_i \coloneqq \{\vv_j\vy^T_i | j=1, \ldots, r \} \subseteq \R^{(r-1) \times (d+1)}$ for $i = 0, \ldots, m$.
Note that $\sum_{i=1}^r \vv_i = \vec{0}$, so $\mathbb{0} \in \conv \, M_i$ for all $i$.
There are $m+1$ such sets in $\R^m$, so \cref{thm:colorful_c} gives us, for each $i$, a point $\vz_i\in M_i$ and a coeficient $\alpha_i\geq 0$ such that 
\begin{equation}\label{eq:col}
 \mathbb{0} = \sum_{i=0}^m \alpha_i \vz_i , \quad \quad \sum_{i=0}^m \alpha_i = 1\, . 
\end{equation}

Let $f:\{0, \ldots, m\} \to \{1, \ldots, r\}$ such that $\vz_i = \vv_{f(i)}\vy_i^T$.
Multiplying by $\vec{e}_t^T$, for $t = 1, \ldots, r-1$, on the left of the first equation of \eqref{eq:col} gives us 
$$ \vec{0}^T = \left( \sum_{i: f(i) = t} \alpha_i \vy_i^T\right) - \left(\sum_{i: f(i) = r} \alpha_i \vy_i^T \right) \, . $$

We therefore get that $ \sum_{i: f(i) = t} \alpha_i \vy_i^T$ does not depend on $t = 1, \ldots, r$.
Recall that $\vy_i = (\vx_i, 1)$, so this gives us that $A \coloneqq  \sum_{i: f(i) = t} \alpha_i$ does not depend on $t$, so $r A = \sum_i \alpha_i = 1$ gives $A = \frac{1}{r}$.

We conclude that $\vv \coloneqq \sum_{i: f(i) = t} \frac{\alpha_i}{A}\vx_i$ does not depend on $t=1, \ldots, r$.
Thus, $\vv \in \conv\{\vx_i\}_{i: f(i) = t}$ for any $t=1, \ldots, r$.

Therefore $\bigcap_t \conv\{\vx_i\}_{i: f(i) = t}\neq \emptyset$, and our desired partition was found.
\end{proof}

\subsection{Borsuk-Ulam theorem}

\begin{thm}[Ham-Sandwish theorem]
Let $X_1 ,  \ldots , X_d \subseteq  \R^d$ be finite sets.
Then there exists a hyperplane $h \subseteq \R^d$ such that for all $i=1, \ldots, d$ we have 
$$ |X_i \cap \overline{h}^+| , |X_i \cap \overline{h}^-| \geq \frac{1}{2} n\, . $$
\end{thm}

\begin{proof}
This claim uses the Borsuk Ulam theorem, to be found in \cite{matouvsek2003using}.
TBD

\end{proof}

\section{Chevalier - Warning}

Commonly used in number theory, the Chevalier-Warning theorem guarantees that there are some solutions to a polynomial equation, provided that there are sufficient variables.
This result has lead to important conjectures in abstract algebra, like Artin's conjecture.

We will state and prove Chevalier - Warning, stated originally in \cite{chevalley1935demonstration}.

\begin{thm}[Chevalier - Warning Theorem]\label{thm:CW}
Fix $q$ a power of a prime $p$, and consider $f_1, \ldots, f_t$ polynomials in $\F_q[\vx_1, \ldots , \vx_m]$ such that $\sum_i \deg (f) < m$.

Then the number of common zeroes of the polynomials $\{ f_i\}_{i=1}^t $ is a multiple of $p$.
\end{thm}

\begin{lm}\label{lm:sum_infield}
Let $q$ be a power of a prime and $r$ a positive integer with $r < q - 1$. 
Then we have
$$ \sum_{x \in \F_q} x^r = 0 \, . $$
\end{lm}

\begin{proof}
We have shown in \cref{lm:rootcount} that any one-variable polynomial of degree $d$ has at most $d$ roots in any field.
Therefore, because $r < q-1$, there is some non-zero $y \in \F_q$ such that $y^r - 1 \neq 0$.
Thus we have 
$$y^r \sum_{x \in \F_q} x^r =  \sum_{x \in \F_q} (yx)^r =  \sum_{x \in \F_q} x^r \, , $$
which implies that $ \sum_{x \in \F_q} x^r = 0$.
\end{proof}

\begin{proof}[Proof of \cref{thm:CW}]
Consider the polynomial 
$$ N(\vx) \coloneqq N(x_1, \ldots, x_m ) \coloneqq \prod_{i=1}^t ( 1 - f_i(\vx )^{q-1} )\, . $$

Recall that $x^{q-1} - 1 $ vanishes in all non-zero field elements $x \in \F_q$.
If $\vx  $ is a common root of $\{f_i\}_{i=1}^t$ then $N(\vx) = 1$, whereas if $\vx $ does not vanish in some $f_i$ then $N(\vx) = 0$.
Thus $\sum_{\vx \in \F_q^m} N(\vx) = \# \{ \text{ common roots of } \{f_i\}_i \}$ modulo $p$.

Note that $\deg N = \sum_{i= 1}^t (q-1)\deg f_i < m(q-1)$, so for any monomial in $N(\vx )$, say $ x_1^{\alpha_1}\cdots x_m^{\alpha_m}$ there is some $i$ such that $\alpha_i < q-1$.
Using \cref{lm:sum_infield} we have 
$$\sum_{(x_1, \ldots, x_m) \in \F_q^m} x_1^{\alpha_1}\cdots x_m^{\alpha_m} = \left(\sum_{x\in\F_q}x^{\alpha_1} \right) \cdots \left(\sum_{x\in\F_q}x^{\alpha_1} \right) = 0\, . $$

Using this on every monomial of $N$, we get $\sum_{\vx \in \F_q^m} N(\vx) = 0$, as desired.
\end{proof}

The following lemma will be useful in applying \cref{thm:CW}.

\begin{lm}\label{lm:power}
Let $q$ be a power of a prime and $x \in \F_q$ a non-zero element.
Then $x^ {q-1} = 1 $.
\end{lm}

\begin{proof}
The set $\F_q\setminus \{0\}$ is a group with the multiplication operation, so Lagrange theorem for groups tells us that the order $t$ of the finite subgroup $\{1, x, x^2, \ldots \} \subseteq \F_q\setminus\{0\} $ divides $|\F_q\setminus \{0\}| = q-1$.

Therefore, $q-1 = tl $ for some integer $l$, so $x^{q-1} = (x^t)^l = 1^l = 1$, as desired.
\end{proof}

\subsection{Berge-Sauer conjecture}

A \textbf{subgraph} $H\leq G$ of a graph $G$ is a graph $H = (V, E)$ such that $V \subseteq V(G)$ and $E \subseteq E(G) $.
Notice that to obtain a subgraph of $G$ we remove some edges and some vertices.

\begin{conj}[Berge-Sauer conjecture]
Every $4$-regular simple graph contains a $3$-regular graph.
\end{conj}

This conjecture is still open, but a slight modification was established in \cite{alon1984every}.
Note that the condition that the graph is simple is necessary: If one takes three vertices with six edges forming three pairs of parallel edges, we get a $4$-regular graph that does not contain a $3$-regular graph.
The theorem below does not require the simple graph assumption.

\begin{thm}\label{thm:3reggraphs}
Every graph $G$ arising from adding an edge $f$ to a $4$-regular graph contains a $3$-regular graph.
\end{thm}

Fix a graph $G = (V(G), E(G))$.
Recall that for a vertex $c\in V(G)$, we write $\deg_G v$ for the number of incident edges at $v$ in the graph $G$.

\begin{proof}[Proof of \cref{thm:3reggraphs}]
For each vertex $v\in V(G)$ consider $f_v \in \F_3[\vx_e | e\in E(G) ]$ defined by
$$f_v(\vx) = \sum_{e\in E(G) \text{ s.t. } v\in e} x_e^2 \, . $$

Note that $\sum_{v \in V(G)} \deg f_v = 2 |V(G)|$.
However, the sum of the incidence of the vertices double counts the edges, so
$$ |E(G)| - 1 = \frac{1}{2} \sum_{v \in V(G)} \deg_{G \setminus f} v = 2 |V(G)| \, . $$
Therefore, $\sum_{v \in V(G)} \deg f_v < |E(G)|$, so we can use \cref{thm:CW}.

That is, the number of common zeroes of $\{f_v\}_{v \in V(G)}$ is a multiple of three.
Since these polynomials have no constant term, they all vanish at $\vx = (0, \ldots, 0)$.
Therefore, there is at least one other common zero $\vy = (\vy_e)_{e \in E(G)}$.
Fix this zero, let $W \subseteq E(G)$ be the set of edges $e \in E(G)$ such that $\vy_e \neq 0$, and let $V' \subseteq V(G)$ be the set of vertices incident to some edge in $W$.
We claim that $H = (V', W)$ is the desired $3$-regular subgraph.

Indeed, using \cref{lm:power}, for any $v \in V'$, $f_v(\vy) = \sum_{\substack{e\in E(G) \\ v \in e}} \vy_e^2 = \sum_{\substack{e\in W \\ v \in e}} 1 = \deg_H(v)$ is zero modulo three.
However $0 < \deg_H(v) \leq \deg_G(v) < 6$, so $\deg_H (v) = 3$.
Furthermore, $V'$ is non empty, as $\vy $ was taken to be non-zero, so $H$ is $3$-regular, as desired.
\end{proof}

\subsection{Aditive number theory}

Let $G$ be an abelian group.
We define 
$$g(G) = \min \left\{ s \in \Z_{\geq 0} \Big| \forall_{g_1, \ldots, g_s \in G} \exists_{I \subseteq [s]} \text{ s.t. } I \neq \emptyset \, \text{ and } \sum_{i \in I} g_i = 0 \right\}\, . $$

\begin{smpl}
One can see that for any integer $n \geq 1 $ we have $g(\Z_n) = n$.
Indeed, $g(\Z_n ) \geq n$ as we can take $g_1 = \cdots = g_{n-1} = 1$ and there is no non-trivial subset of $[n-1]$ that results in a zero sum.

On the other hand, if $g_1, \ldots, g_n \in \Z_n $ then in the following list
$$ g_1, g_1 + g_2, \ldots, g_1 + \cdots + g_n \, , $$
there is either a zero or a repeated value of $\Z_n$.
In either case, we can find a sum that vanishes, so $g(\Z_n) \geq n$.
\end{smpl}

The following generalization requires the use of Chevalier - Warning theorem.

\begin{thm}[Olsen's theorem]
Let $p$ be a prime number and $G = \Z_p^{\oplus k} = \underbrace{\Z_p \oplus \cdots \oplus \Z_p}_{\text{ $k$ times } }$.
Then $g(G) = k(p-1)+ 1$.
\end{thm}

\begin{proof}
One can see that for any integer $k \geq 1 $ we have $g(\Z_p^{\oplus k}) \geq k(p-1)+1$.
Indeed, take the following  $g_{j, i} = \ve_j$, the $j$-th element of the canonical basis, for $i=1, \ldots, p-1$ and $j=1, \ldots, k$.
There is no non-trivial subset that results in a zero sum:
assume otherwise, so that there are $a_1, \ldots, a_k \in \{0, \ldots, p-1\}$ with $0 \leq a_j \leq p-1$ for which $\sum_{j=1}^k a_j g_{j, 1} = (0, \ldots,  0) $.
Then $a_j = 0$ modulo $p$, so this corresponds to the empty set.

For the remaning inequality, that $g(\Z_p^k) \leq k(p-1)+1$, fix $k(p-1)+1$ elements $g^{(0)}, \ldots, g^{(k(p-1))} \in \Z_p^k$, and consider the following polynomials $f_i \in \F_p[x_0, \ldots, x_{k(p-1)}]$ for $i = 1, \ldots, k$:
$$f_j(\vx ) \coloneqq \sum_{i=0}^{k(p-1)} x_i^{p-1}g^{(i)}_j \, .$$

Note how $\sum_{j=1}^ k \deg(f_j) = k(p-1)$ is less than the number of variables, so the number of common zeroes is a multiple of $p$.
The all-zero vector is a common zero, so there is some non-trivial solution $\vy$ such that $f_j(\vy) = 0$ for $j=1, \ldots, k$.

Let $L \coloneqq \{ i\in\{0, \ldots, k(p-1)\} | \vy_i \neq 0 \}$.
The condition $f_j(\vy) = 0$, together with \cref{lm:power} gives us $\sum_{i\in L} g^{(i)}_j = 0$.
Since $j$ is generic, we get that $\sum_{i\in L} g^{(i)} = 0$ in $\Z_p^k$.
As $\vy $ was chosen to be non-trivial, the set $L$ is non-empty, so this is the desired set.
\end{proof}

\begin{thm}[Erd\"os-Ginzburg-Ziv Theorem]
Let $a_1, \ldots, a_{2n-1} \in \Z_n $.
Then there exists a set $I\subseteq [2n-1] $ of size $n$ such that $\sum_{i\in I} a_i = 0$.
\end{thm}

For a generic aditive group $G$, we define
$$ f(G, n) = \min \left\{ s \in \Z_{\geq 0} \Big| \forall_{f_1, \ldots, f_s \in G} \exists_{I \subseteq [s]} \text{ s.t. } |I| = n \, \text{ and } \sum_{i \in I} f_i = 0 \right\} \, ,$$
so that the theorem becomes $ f(\Z_n, n) = 2n - 1$.

\begin{proof}
As usual, we provide a multiset in $\Z_n$ of size $2n-2$ to show that $f(\Z_n, n) \geq 2n-1$.
Specifically, let $f_{(0, j)} = 0$ and $f_{(1, j)} = 1$ for $j=1, \ldots, n-1$.
There is no subset $I$ with $n$ elements with zero sum:
assume otherwise, so that there are $a_0, a_1 \in \{0, \ldots, n-1\}$ with $0 \leq a_j \leq n-1$ for which $a_0 + a_1 = n$ and $ a_1 = 0$ modulo $n$.
This with $0 \leq a_1 \leq n-1$ gives $a_1 = 0$, so $a_0 = n$, which is impossible.

For the remaning inequality, that $f(\Z_n, n) \leq 2n - 1$, we act by induction on the number of prime factors of $n$, counting multiplicity.
That is, if $n = \prod_i p_i^{\alpha_i}$ is the factorization of $n$ into primes, then our induction step is on $\sum_i \alpha_i$.

The base case is when $\sum_i \alpha_i = 1$, that is when $n$ is prime.
So let us assume that $n$ is a prime number.
Fix $2n - 1$ elements $f_{0}, \ldots, f_{2n  - 2} \in \Z_n$, and consider the following polynomials $a_1, a_2 \in \F_p[x_0, \ldots, x_{2n  - 2 }]$:
\begin{align*}
a_1(\vx ) &\coloneqq \sum_{i=0}^{2n-2} x_i^{n-1}f_i \, , \\
a_2(\vx ) &\coloneqq \sum_{i=0}^{2n-2} x_i^{n-1} \, .
\end{align*}

We have that $\deg a_1 + \deg a_2 = 2n-2 < 2n-1$, which is the number of variables.
Thus \cref{thm:CW} tells us that the number of common roots of $a_1, a_2$ is a multiple of $n > 1$.
Furthermore, the all-zero vector is a common root, so there is a non-trivial common root $\vy = (y_0 , \ldots, y_{2n-2}) \in \Z_n^{2n-1}$.

Let $L \coloneqq \{ i\in\{0, \ldots, 2n - 2\} | y_i \neq 0 \}$.
The condition $a_2(\vy) = 0$, together with \cref{lm:power} gives us $\sum_{i\in L} 1 = 0$.
Therefore $|L|\leq 2n-1$ is a multiple of $n$.
As $\vy $ was chosen to be non-trivial, the set $L$ is non-empty, so $|L| = n$.
The condition $a_1(\vy) = 0$, together with \cref{lm:power} gives us $\sum_{i\in L} f_i = 0$, so this is the desired set.

Now for the induction step.
If $n$ is not prime, then $n = mt$ where $m, t > 1$ and by induction hypothesis $f(\Z_m, m) = 2m - 1$ and $f(\Z_t, t) = 2t - 1$.
We wish to show that $f(\Z_n, n) \leq 2n-1$, so we consider $2n - 1$ elements $f_{1}, \ldots, f_{2n  - 1} \in \Z_n$.
Because $n$ is a multiple of $m$, we can analyse elements of $\Z_n$ modulo $m$, which defines a projection $\pi: \Z_n \to \Z_m$ that is a group homomorphism.
Let $g_i = \pi(f_i)$ of $i = 1, \ldots, 2n - 1$.

Let $T_1 = [2n-1]$ and for $i = 1, \ldots, 2t-1$ define iteratively
\begin{align*}
I_i &\subseteq T_i \text{ such that $|I_i| = m$ and $\sum_{j\in I_i} g_j = 0$,} \\
T_i &\coloneqq T_{i-1} \setminus I_{i-1} \, .
\end{align*}

Note that $|T_i| = |T_1| - (i-1)m \geq  |T_1| - ((2t-1) - 1)m = 2n - 1 - 2tm  + 2m = 2m - 1$, so such an $I_i \subseteq T_i $ exists by induction hypothesis (remember that $f(\Z_m, m) = 2m-1$).
For $j=1, \ldots, 2t-1$, define $h_j m \coloneqq  \sum_{i \in I_j} f_i$.
This is possible because we know that $ \sum_{i \in I_j} f_i = 0 $ modulo $m$.
Furthermore, we can identify $h_j $ with an element in $\Z_t$.
By induction hypothesis, there is some set $I \subseteq [2t-1]$ such that $|I| = t$ and $\sum_{j\in I} h_j = 0 $ modulo $t$.

Thus $\sum_{j\in I} \sum_{i \in I_j} f_i = 0$ gives us a sum of $mt = n$ terms that is zero, so this concludes the induction step.
\end{proof}

\subsection{Kemnitz conjecture}

The natural generalisation of Erd\"os-Ginzburg-Ziv Theorem requires a higher dimensional group.
Specifically, for integers $n, k$, we can compute $f(\Z_n^{\oplus k}, n)$.
The answer, for $k=2$, was conjectured by Kemnitz in 1983 and proved in \cite{reiher2007kemnitz}.

\begin{thm}[Reiher's Theorem]\label{thm:RT}
$$f(\Z_n^2, n ) = 4n - 3 \, . $$
\end{thm}

\begin{lm}
$$f(\Z_2^2, 2 ) = 5\, . $$
\end{lm}

\begin{proof}
This proof is not presented here.
\end{proof}

\begin{lm}
The set $L = \{n \in \Z_{\geq 0} | f(\Z_n^2, n ) = 4n - 3 \} $ is multiplicative.
That is, if $m, n\in L$ then $mn\in L$.
\end{lm}

\begin{proof}
This proof is not presented here.
\end{proof}

\begin{thm}[Lucas' theorem]
Let $p$ be a prime number, and $a, b $ positive integers along with their expansion in base $p$ given by:
$$a = \sum_{i=0}^M a_ip^i, \, \, b = \sum_{i=0}^M b_ip^i \, ,$$
that is, $M\geq 0 $ is an integer and $a_i, b_i \in \{ 0 , \ldots , p-1 \}$ for $i=0, \ldots M$.

Then 
$$\binom{a}{b} = \binom{a_M}{b_M} \cdots \binom{a_1}{b_1} \binom{a_0}{b_0} \text{ modulo $p$.}$$
\end{thm}

\begin{proof}
This proof is not presented here.
\end{proof}

\begin{thm}[Fermat's little theorem]
For $a \in \Z$ and $p$ prime, we have
$$a^p = a \text{ modulo $p$.}$$
\end{thm}

\begin{proof}
This proof is not presented here.
\end{proof}

We now introduce useful notation for the proof of the Kemnitz conjecture.
Let $p$ be a prime number, $J, K$ be a multisets of elements in $\Z_p$ and $x$ an integer.
We denote $\sum K $ for $\sum_{k\in K} k \in \Z_p$, and we denote $(x|J)$ for the number of multisets $I\subseteq J$ such that $|I| = x$ and $\sum I = 0$.

\begin{lm}
Let $p$ be a prime number, and let $J$ be a multiset of elements in $\Z_p$.
\begin{equation}
\text{If $|I| = 3p - 3$, we have } 1 - (p-1|J) - (p|J) + (2p-1|J) + (2p|J) = 0 \text{ modulo $p$,}
\end{equation}
\begin{equation}
\text{If $|I| \in \{ 3p - 2, 3p - 1\} $, we have } 1 - (p|J) + (2p|J) = 0 \text{ modulo $p$,}
\end{equation}
\begin{equation}
\text{If $|I| = 4p - 3$, we have } 1 - (p|J) + (2p|J) - (3p|J) = 0 \text{ modulo $p$,}
\end{equation}
\begin{equation}
\text{If $|I| = 4p - 3$, we have } (p-1|J) - (2p-1|J) + (3p-1|J) = 0 \text{ modulo $p$,}
\end{equation}
\begin{equation}
\text{If $|I| = 4p - 3$, we have } 3 - 2 (p-1|J) - 2 (p|J) + (2p-1|J) + (2p|J) = 0 \text{ modulo $p$,}
\end{equation}
\begin{equation}
\text{If $|I| = 4p - 3$ and $(p|J) = 0$, we have } (p-1|J) = (3p-1|J) \text{ modulo $p$,}
\end{equation}
\end{lm}

\begin{proof}
This proof is not presented here.
\end{proof}

\begin{proof}[Proof of \cref{thm:RT}]
This proof is not presented here.
\end{proof}

\section{Combinatorial Nullstelensatz}

This section is motivated by a classical result in abstract algebra, shown in \cite{hilbert1893ueber}.

\begin{thm}[Nullstellensatz]
Let $\F$ be an \textbf{algebraically closed field}, and $f, g_1, \ldots, g_m\in \F[\vx_1, \ldots, \vx_n]$ polynomials.
If $f$ vanishes in all $\vv$ that are common roots of $g_1, \ldots, g_m$, then there exists an integer $k$ and polynomials $h_1, \ldots, h_m \in \F[\vx_1, \ldots, \vx_n]$ such that
$$ f^k = \sum_{i=1}^m h_i g_i \, . $$
\end{thm}

We will be interested in its combinatorial versions, with applicatios in Aditive number theory and chromatic invariants.

\begin{thm}[Combinatorial Nullstellensatz - I]\label{thm:NSS-1}
Let $q$ be a power of a prime and $f\in \F_q[x_1, \ldots, x_n]$.
Take $S_1, \ldots, S_n \subseteq \F_q$ such that $f$ vanishes in $S_1\times \cdots \times S_n$.

Define $g_i\coloneqq \prod_{s \in S_i} (x_i - s)$ for $i=1, \ldots, n$.
Then, there exist polynomials $h_1, \ldots, h_n \in \F_q[x_1, \ldots, x_n]$ such that $\deg h_i \leq \deg f - \deg g_i$ and 
$$f = \sum_{i=1}^n h_i g_i \, .$$
\end{thm}

If $f\in \K[\vx]$ is a polynomial, we denote $[\vx^\alpha] f\in \K$ for the coefficient of $f$ in the monomial $\vx^{\alpha}$.

\begin{thm}[Combinatorial Nullstellensatz - II]\label{thm:NSS-2}
Let $q$ be a power of a prime.
Let $f \in \F_q[x_1, \ldots, x_n]$, consider sets $S_i \subseteq \F_q$ and consider non-negative integers $t_1, \ldots, t_n$ such that $\deg f = \sum_{i=1}^n t_i$, $|S_i| > t_i$ for $i=1, \ldots, n$ and $\left[ \prod_{i=1}^n x_i^{t_i} \right] f \neq 0$.

Then there exists $\vy \in S_1\times \cdots \times S_n$ such that $f(\vy) \neq 0$.
\end{thm}

For a multivariate polynomial $f \in \K[x_1, \ldots, x_n]$, we denote $\deg_{x_i} f $ for the largest exponent of $x_i$ in any monomial of $f$, and $\deg f$ is the total degree of $f$.

\begin{lm}\label{lm:large_vanishing_poly}
Let $q$ be a power of a prime.
Let $f\in \F_q[x_1, \ldots, x_n]$, $t_i$ non-negative integers and $S_i \subseteq \F_q$ for $i=1, \ldots, n$ such that $\deg_{x_i} f \leq t_i < |S_i|$.
If $f$ vanishes on $S_1\times \cdots \times S_n$, then $f$ is the zero polynomial.
\end{lm}

\begin{proof}
We act by induction on $n$, the number of variables.
For $n = 1$, assume for sake of contradction that $f$ is not the zero polynomial.
Note that we have $\deg f = \deg_{x_1} f \leq t_1$, so \cref{lm:rootcount} gives us that $f$ vanishes in at most $t_1 $ points of $\F_q$.
But $|S_1| > t_1$, so if $f$ vanishes on $S_1$ we get a contradiction.

For the induction step, the assumption $\deg_{x_n} f \leq t_n$ allows us to write
$$f(x_1, \ldots, x_n) = \sum_{j=0}^{t_n} g_j(x_1, \ldots, x_{n-1}) x_n^j \, ,$$
for some polynomials $g_i \in \F_q[x_1, \ldots, x_{n-1}]$.

Assume that $f$ vanishes on $S_1\times \cdots \times S_n$.
For each $(s_1, \ldots, s_{n-1}) \in S_1\times \cdots \times S_{n-1}$, the one variate polynomial $f(s_1, \ldots, s_{n-1}, x) \in\F_q[x]$ has degree at most $t_n$ and vanishes in $S_n$, which has more than $t_1$ elements, so it must be the zero polynomial.
This concludes that $g_j(s_1, \ldots, s_{n-1}) = 0$, that is each $g_j$ vanishes in $S_1\times \cdots \times S_{n-1}$ and has $\deg_{x_i} g_j \leq t_i < |S_i|$.
By induction hypothesis, $g_j$ must be the zero polynomial, so $f $ is also the zero polynomial, concluding the proof.
\end{proof}

\begin{proof}[Proof of \cref{thm:NSS-1}]
We can assume with no loss of generality that $|S_i| = t_i + 1$, as shrinking the sets $S_i$ does not break the theorem assumptions.
Note that $g_i(\vx) = \prod_{s\in S_i} (x_i - s) = x_i^{t_i + 1} + p_i(x_i)$, where $p_i$ is a one variable polynomial of degree at most $t_i$.

By iteratively applying the division algorithm $n$ times, there are polynomials $h_1, \ldots, h_n$ and $\hat{f}$ such that 
$$f = \sum_{i=1}^n h_i g_i + \hat{f}, \, \, \, \, \, \text{ where $\deg_{x_i}\hat{f} \leq t_i$ and $\deg h_i \leq \deg f - \deg g_i$ for $i=1, \ldots, n$.} $$

Note how both $f$ and $ \sum_{i=1}^n h_i g_i $ vanish on $S_1\times \cdots \times S_n$, so $\hat{f}$ also vanishes on $S_1\times \cdots \times S_n$.
Thus, from \cref{lm:large_vanishing_poly}, $\hat{f}$ is the zero polynomial, and we get the desired expression.
\end{proof}

\begin{proof}[Proof of \cref{thm:NSS-2}]
We can assume with no loss of generality that $|S_i| = t_i + 1$, as shrinking the sets $S_i$ does not break the theorem assumptions.
For sake of contradiction, assume that $f$ vanishes on $S_1\times \cdots \times S_n$.
From \cref{thm:NSS-1}, there are polynomials $h_1, \ldots, h_n$ with $\deg h_i \leq \deg f - \deg g_i $ and $f = \sum_{i=1}^n h_ig_i$.

Note that $\deg h_i \leq \deg f - \deg g_i < \sum_{\substack{j = 1 \\ j \neq i}}^n t_j$ for $i=1, \ldots, n$.
Let $ \vx^\mathbf{t} \coloneqq \prod_{i=1}^n x_i^{t_i} $.
Fix $i$, so that we have
$$ \left[  \vx^\mathbf{t} \right]  h_i g_i  = \sum_{\substack{\vx^\alpha \vx^\beta = \vx^\mathbf{t}}} [\vx^\alpha]h_i [\vx^\beta]g_i \, . $$

The condition $\vx^\alpha \vx^\beta = \vx^\mathbf{t}$ tells us $ \sum_j \alpha_j + \sum_j \beta_j = \sum_j t_j$.
If $\sum_j \alpha_j \geq \sum_{\substack{j = 1 \\ j \neq i}}^n t_j$ we have $[\vx^\alpha]h_i = 0 $, and if  $\sum_j \alpha_j  < \sum_{\substack{j = 1 \\ j \neq i}}^n t_j$ we have $ \sum_j \beta_j > t_i$, in which case $\vx^\alpha \vx^\beta = \vx^\mathbf{t} $ is impossible.

Therefore 
$$ \left[  \vx^\mathbf{t} \right]  h_i g_i  = \sum_{\substack{\vx^\alpha \vx^\beta = \vx^\mathbf{t}}} [\vx^\alpha]h_i [\vx^\beta]g_i = 0\, , $$
which is a contradiction with the original assumption that $  \left[  \vx^\mathbf{t} \right] f \neq 0$.
This concludes the proof.
\end{proof}

\subsubsection*{Aditive group theory}

For $A, B\subseteq \Z_n$, write $A+B \coloneqq \{a+b | a\in A, \, b\in B\}$ and write $A\hat{+}B \coloneqq \{a+b | a\in A, \, b\in B, \, a \neq b\}$.

\begin{thm}[Cauchy-Davenport theorem]
Let $p$ be a prime number and $A, B\subseteq \Z_p$ non-empty sets, then $|A+B| \geq \min \{p, |A| + |B| - 1\}$.
\end{thm}

\begin{proof}
If $|A| + |B| > p$, then for any $z \in \Z_p$ we have 
$$ |A \cap (z - B)| = |A| + |z-B| - |A \cup (z-B)| \geq |A| + |B| - p > 0 \, ,$$
so there is some $x \in A \cap (z - B)$, so we have $z \in A + B$.
Since $z$ was generic, we conclude that $\Z_p \subseteq A+B$ and $|A+B| \geq p$.

If $|A| + |B| \leq p$, assume for sake of contradiction that $|A+B| < |A| + |B| - 1$ and let $C \subseteq \Z_p$ such that $A+B \subseteq C $ and $|C| = |A| + |B| - 2$.
Define the polynomial $f \in \Z_p[x, y]$ of degreee $|A|+|B|-2$ given by:
$$ f(x, y) \coloneqq \prod_{c \in C} (x + y - c) \, . $$

Let $t_1 = |A| - 1$ and $t_2 = |B| -1 $.
Note that $f$ vanishes in $A\times B$.
From \cref{thm:NSS-2} we have that $[x^{t_1} y^{t_2}] f = 0 $.

On the other hand, $\deg f = t_1+t_2$ so we can compute 
$$[x^{t_1} y^{t_2}] f = [x^{t_1} y^{t_2}]\prod_{c \in C} (x + y - c) = [x^{t_1} y^{t_2}] (x+y)^{|C|} = \binom{|C|}{t_1} \, , $$
which is not a multiple of $p$, as $t_1, t_2 > 0$.
We get a contradiction.
\end{proof}

The following result was established in \cite{da1994cyclic}, solving a conjecture from Erd\"os and Heilbronn from 1966.

\begin{thm}[Silva-Hamidoune Theorem]
Let $p$ be a prime number and $A\subseteq \Z_p$ a non-empty set, then $|A\hat{+}A| \geq \min \{p, 2|A| - 3\}$.
\end{thm}

\begin{proof}
If $2|A| > p + 3$, then for any $z \in \Z_p$ we have 
$$ |A \cap (z - A)| = |A| + |z-A| - |A \cup (z-A)| \geq 2|A| - p > 3 \, ,$$
so there are distinct $x_1, x_2 \in A \cap (z - B)$, so we have $z \in A + B$.
Since $z$ was generic, we conclude that $\Z_p \subseteq A+B$ and $|A+B| \geq p$.

If $|A| + |B| \leq p$, assume for sake of contradiction that $|A+B| < |A| + |B| - 1$ and let $C \subseteq \Z_p$ such that $A+B \subseteq C $ and $|C| = |A| + |B| - 2$.
Define the polynomial $f \in \Z_p[x, y]$ of degreee $|A|+|B|-2$ given by:
$$ f(x, y) \coloneqq \prod_{c \in C} (x + y - c) \, . $$

Let $t_1 = |A| - 1$ and $t_2 = |B| -1 $.
Note that $f$ vanishes in $A\times B$.
From \cref{thm:NSS-2} we have that $[x^{t_1} y^{t_2}] f = 0 $.

On the other hand, $\deg f = t_1+t_2$ so we can compute 
$$[x^{t_1} y^{t_2}] f = [x^{t_1} y^{t_2}]\prod_{c \in C} (x + y - c) = [x^{t_1} y^{t_2}] (x+y)^{|C|} = \binom{|C|}{t_1} \, , $$
which is not a multiple of $p$, as $t_1, t_2 > 0$.
We get a contradiction.
\end{proof}

\subsubsection{Hyperplane arrangements}

\begin{thm}
Let $n$ be a positive integer, $\vec{0}$ be the all-zero vector in $\R^n$ and let $H = \{0, 1\}^n\setminus \{\vec{0}\}$.
Assume that $H_1, \ldots, H_m$ are hyperplanes in $\R^n$ such that $H \subseteq \bigcup_{i=1}^n H_i$ and $\vec{0}\not\in \bigcup_{i=1}^n H_i$, 
Then $m \geq n$.
\end{thm}

\begin{proof}
This proof is not presented here.
\end{proof}

\section*{Aknowledgments}
The author is supported by the Max Planck institute for the sciences. 
These notes are based on a lecture by Benny Sudakov at ETH, in 2014.
These notes benefited tremendously from coments and suggestions of students from Algebraic Methods in Combinatorics 2023.

\bibliographystyle{alpha}
\bibliography{bibli}

\end{document}